\documentclass[11pt]{article}
\usepackage{graphicx, color, xcolor}
\usepackage{amsmath,amsfonts,amssymb}

\def\forall{\hbox{for all}~}
\def\L{\mathbf{L}}

\def\bfX{{\bf X}}

\def\ve{\varepsilon}

\def\R{{\mathbb R}}%{I\!\!R}

\def\implies{\Longrightarrow}
\def\vp{\varphi}

\def\P{{\cal P}}
\def\tv{\hbox{Tot.Var.}}

\def\v{\vskip 1em}
\def\O{{\cal O}}
\def\M{{\cal M}}
\def\C{{\cal C}}
\def\F{{\cal F}}
\def\S{{\cal S}}
\def\prob{\hbox{Prob.}}

\def\ov{\overline}
\def\Tilde{\widetilde}
\def\Hat{\widehat}
\def\meas{\hbox{meas}}

\def\dint{\int\!\!\int}

\def\bega{\begin{array}}
\def\enda{\end{array}}
\def\begi{\begin{itemize}}
\def\endi{\end{itemize}}
\def\TV{\hbox{TV}}
\def\Var{\hbox{Var}}

\def\ds{\displaystyle}
\def\bel{\begin{equation}\label}
\def\eeq{\end{equation}}
\def\sqr#1#2{\vbox{\hrule height .#2pt
\hbox{\vrule width .#2pt height #1pt \kern #1pt
\vrule width .#2pt}\hrule height .#2pt }}
\def\square{\sqr74}
\def\endproof{\hphantom{MM}\hfill\llap{$\square$}\goodbreak}

\newtheorem{theorem}{Theorem}[section]

\newtheorem{lemma}{Lemma}[section]

\newtheorem{proposition}{Proposition}[section]
\newtheorem{remark}{Remark}[section]
\newtheorem{definition}{Definition}[section]

\definecolor{bluegreen}{rgb}{0.0, 0.87, 0.87}
\definecolor{blush}{rgb}{0.87, 0.36, 0.51}
\definecolor{bittersweet}{rgb}{1.0, 0.44, 0.37}
\definecolor{burgundy}{rgb}{0.5, 0.0, 0.13}
\definecolor{cardinal}{rgb}{0.77, 0.12, 0.23}

\begin{document}

\title{\bf A Two-Fluxes Stochastic Model of Traffic Waves}

\author{Alberto Bressan and Sumantha Kanale Suresha
\\    \, \\
{\small Mathematics Department, Penn State University,} \\
{\small University Park, PA 16802, U.S.A. } \,\\ 
\,\\
{\small E-mails: axb62@psu.edu,~~sfk6107@psu.edu}
}
\maketitle

\begin{abstract}
The paper introduces a stochastic model for the spontaneous formation of 
traffic waves on a highway.   This is formulated in terms of a conservation law with discontinuous, gradient-dependent flux.
In an unstable regime, the non-uniqueness of solutions allows for the emergence of $N$-shaped
spikes in the traffic density, at random points and times.  
Bounds are proved on the expected value of the total variation of the random solution and on the 
expected number of shocks. Further bounds are obtained on the
average velocity and on the expected average acceleration of cars, along a given stretch of highway. Finally, it is proved that the Markov process, whose paths are random solutions to the conservation law,  admits  a unique stationary probability distribution and is ergodic. 
\end{abstract}

\maketitle
%\tableofcontents
%%%%%%%%%%%%%%%%%%%%



\section{Introduction}
\label{sec:1}
\setcounter{equation}{0}

The celebrated LWR  model of traffic flow \cite{LW, R} 
\bel{LWR}
\rho_t + \bigl(\rho v(\rho)\bigr)_x~=~0,\eeq
has provided a basic tool for the analysis of vehicular traffic on a highway.
Here $\rho=\rho(t,x)$ is the traffic density, while   $v=v(\rho)$ is the speed of cars. 
The above macroscopic  model is effective in a variety of situations, and has 
inspired a vast literature.
We remark, however, that a single conservation law 
\bel{claw}
u_t + f(u)_x~=~0,\qquad\qquad u(0,x)=\bar u(x)\eeq
cannot account for the spontaneous 
formation of stop-and-go waves.  Indeed, as shown in Fig.~\ref{f:df125}, left,
if the initial condition satisfies
$$ u^-\leq \bar u(x)\leq u^+\qquad\qquad\forall x\in \R,$$
a comparison argument implies 
$$ u^-\leq u(t,x)\leq u^+\qquad\qquad\forall x\in \R, ~t\geq 0.$$
In particular, if $u(0,\cdot)$ is nearly constant, the solution will never develop the large oscillations shown in  Fig.~\ref{f:df125}, right.

\begin{figure}[ht]
\centerline{\hbox{\includegraphics[width=15cm]{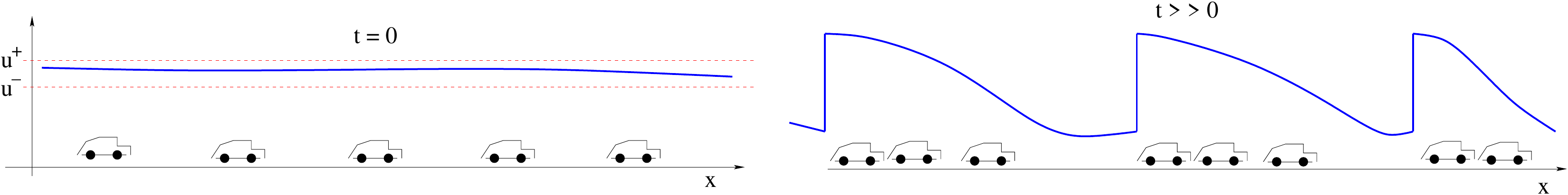}}}
\caption{\small Left: a nearly constant initial density.   Right: the traffic density in the presence of stop-and-go waves.  For a solution to the conservation law (\ref{claw}), by a comparison argument this configuration cannot be reached from nearly constant initial data.}
\label{f:df125}
\end{figure}

Aim of the present paper is to introduce a ``simplest possible" macroscopic model of traffic flow, which accounts for the random creation of stop-and-go waves.

Motivated by the models introduced in \cite{ABS1, ABS2}, we fix $\delta\geq 0$ and 
consider solutions of 
\bel{fgclaw}
u_t + \Big[\theta(u_x) f(u) + \bigl(1-\theta(u_x) \bigr)g(u)\Big]_x=~0,\qquad \quad
\theta(s)~=~\left\{\bega{rl} 1\quad &\hbox{if}\quad s>-\delta,\cr
0\quad &\hbox{if}\quad s<-\delta,\enda\right.\eeq
which are spatially periodic  with period $L$.
In this case,
$u=u(t,x)$ is a function defined for   $t\geq 0$ and $x\in [0,L]$, with 
\bel{uper} u(t,0) ~=~u(t, L)\qquad\qquad\forall t\geq 0.\eeq

On the two fluxes $f,g$ we assume (see Fig.~\ref{f:df131}, right)
\begi
\item[{\bf (A1)}]{\it 
The flux  $g$ is a quadratic polynomial:
\bel{gprop}
g(u)~=~ {\frac{\beta}{2}} u (\alpha-u) ,\eeq
for some constants $\alpha,\beta>0$.
Moreover,  $f$ is a $\C^2$ concave flux such that 
\bel{fprop}\left\{  \bega{l}f(0)\,=\, g(0)\,=\,f(\alpha) \,=\,  g(\alpha)\,=\,0,\\[2mm]
f'(0)=g'(0),\quad f'(\alpha) = g'(\alpha),\enda\right.
\qquad\quad  f(u)>g(u)\quad\forall u\in \, ]0,\alpha[\,.
\eeq
We also assume that, for some $\beta^*>0$, there holds
\bel{fgn}
f''(u)\,\leq-\beta^*\,,\qquad g''(u)\,\leq -\beta^*\qquad\qquad\forall u\in \,]0,\alpha[\,.\eeq
}
\endi 
Notice that, by (\ref{gprop}), we must have $0<\beta^*\leq\beta= - g''(u)$.

We consider a constant initial data, say
\bel{id} u(0,x)~=~ u_0 \in ~]0,\alpha[\qquad\forall x\in [0,L].\eeq

Of course, this Cauchy problem admits the constant solution 
\bel{sol0}u(t,x) = u_0\qquad\qquad\forall (t,x)\in \R_+\times [0,L].\eeq
Since in this case $u_x \equiv 0 > -\delta$, for this particular solution by (\ref{fgclaw}) the flux is $f(u_0)$.
In addition, as observed in \cite{ABS2}, there are infinitely many other solutions where new spikes arise at random points $(t_j, X_j)$ in time and space.  In this paper, we construct a stochastic process by 
introducing a probability measure on the family of all such solutions.

\begin{figure}[ht]
\centerline{\hbox{\includegraphics[width=12cm]{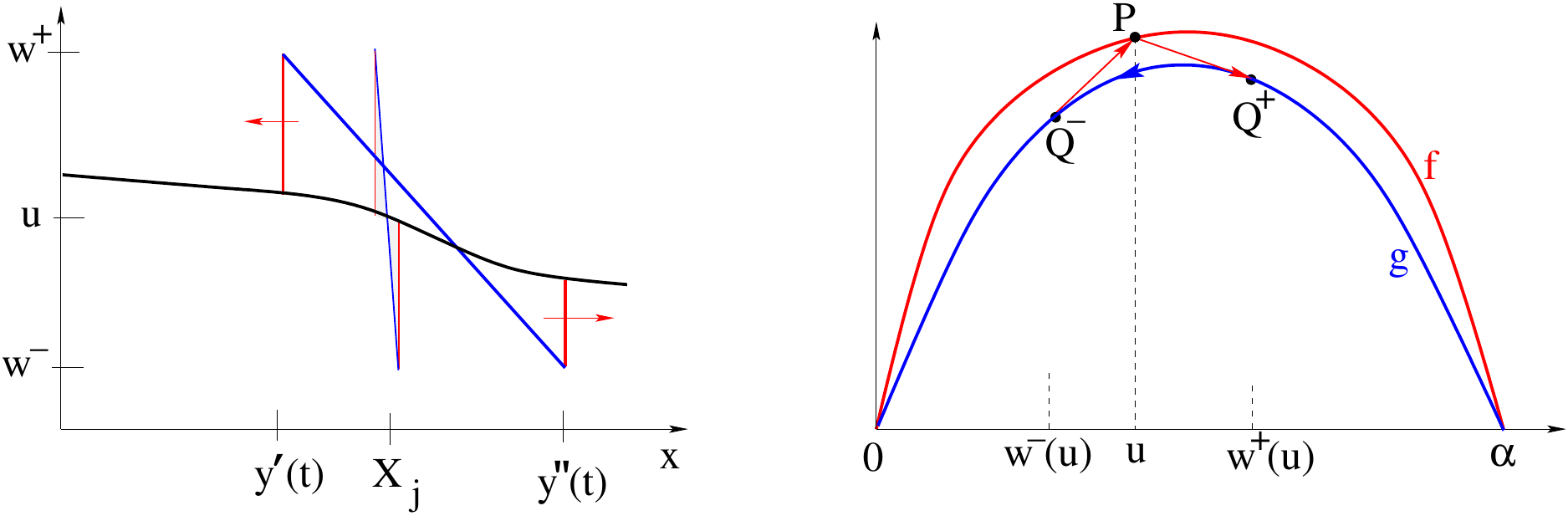}}}
\caption{\small  On a region where the flux is $f$, a new $N$-shaped spike
can arise at any time $t_j$. This consists of a (decreasing)  centered rarefaction wave enclosed between two upward shocks.
For $t>t_j$ the two new shocks, located at $y'(t)<y''(t)$, initially travel with speeds 
given by the slopes of the two secant lines through $P$ which are tangent to the graph of $g$ at $Q^+$, and $Q^-$.  }
\label{f:df131}
\end{figure}

Let \bel{prt} 0~<~t_1~<t_2~<~t_3~<~\cdots\eeq
be a Poisson arrival process with rate $\lambda$.
That means: the random times $t_j$ have independent, identically distributed increments. Setting $t_0=0$,
for every $j\geq 1$ there holds
\bel{P}\hbox{Prob.}\{ t_j- t_{j-1} > s\}~=~e^{-\lambda s}\qquad\qquad\forall s\geq 0.\eeq
In addition, we consider a sequence of independent random variables $X_j$, uniformly distributed over the interval 
$[0,L]$.

Given a sequence $(t_j,X_j)_{j\geq1},$ we construct a solution to the spatially periodic initial-value problem (\ref{fgclaw}), (\ref{id}),  according to the following rules.
\begi
\item[{\bf (I)}] On the first time interval $[0, t_1]$, the solution is constant:
$u(t,x) = u_0$ for all $x\in [0,L]$. 
\item[{\bf (II)}] Assume that the solution has been constructed on $[0, t_j]$.   Given the random point $X_j\in [0,L]$,
the solution is extended to $[0,t_{j+1}]$ as follows:
\begi
\item
If $u_x(t_j, X_j)>-\delta$, so that
$\theta\bigl(u_x(t_j, X_j)\bigr)=1$ and the local flux function  is $f$, then the solution is prolonged for $t>t_j$
by inserting a new $N$-shaped spike at the point $X_j$ at time $t_j$.
This will be supported on an expanding interval
$\bigl[y'(t), y''(t)\bigr]$ (see Fig.~\ref{f:df131}, left).
\item If the gradient $u_x(t_j, X_j)$ is not defined, or if $u_x(t_j, X_j)\leq -\delta$, 
 then the solution is prolonged 
without adding any new spike.
\endi
\endi

Under the assumptions {\bf (A1)}, for every sequence 
$(t_j, X_j)_{j\geq 1}$ we shall determine a unique spatially periodic solution
to the conservation law (\ref{fgclaw}). Having assigned a probability distribution
on the set of all such sequences, we thus obtain a stochastic process whose
random paths $t\mapsto u(t,\cdot)\in BV\bigl([0,L]\bigr)$ are solutions to (\ref{fgclaw}).
Our main goal is to analyze this Markov process and derive some of its basic properties.

The remainder of the paper is organized as follows.  In Section~\ref{sec:2} we recall the definition of weak solution and prove that, for any sequence of random points 
$(t_j, X_j)_{j\geq 1}$, a unique weak solution can be determined, with suitable properties.
In Section~\ref{sec:3} we show that, in the special case where $\delta=0$, as 
considered in 
\cite{ABS2}, all random solutions converge to the constant function $\bar u(x)=u_0$ 
as $t\to +\infty$.   To avoid this trivial dynamics, in all subsequent sections 
we thus assume $\delta>0$.

Section~\ref{sec:4} establishes upper bounds on the expected value and on the variance
of the total variation
of these random solutions.  In Section~\ref{sec:5}, an application of Jensen's
inequality shows that the average speed of cars can never exceed the constant 
$f(u_0)/u_0$.  Moreover, from the analysis of the total variation, 
an upper bound on the expected average acceleration of cars is obtained.
Notably,  all our upper bounds are independent of time.

For subsequent use, in Section~6 we focus on a particular type of solutions.
Starting from any initial data, Lemma~\ref{l:61} shows that with positive probability
after some fixed time $T$ the solution contains exactly one single shock.  Furthermore, the
family of all these 1-shock profiles is a 2-dimensional manifold, i.e., it depends smoothly on just two parameters.

Relying on Lemma~\ref{l:61}, an upper bound on the expected number of shocks in a random solution is then derived in Section~\ref{sec:7}.   We observe that the total number of shocks increases  when a new spike 
appears, but decreases  whenever two shocks merge into each other.

In Section~\ref{sec:8} we prove that our stochastic process
admits a unique stationary probability measure.   Moreover, it is ergodic.
These results are achieved by checking that the conditions 
in \cite{HM, H} are satisfied.  Again, the proof relies in an essential way on 
Lemma~\ref{l:61}. Finally,
in Section~\ref{sec:9} we present the results of some numerical simulations.
Possible extensions and open problems are discussed in Section~\ref{sec:10}.

The modeling and control of stop-and-go traffic waves has been the subject
of an increasing amount of literature, see in particular~\cite{APT, BMPR, CPT, CDGP,
GGLP, Seib, Li, RRS, XXX}.
Our study has been partly motivated by the hysteretic approach of A.\,Corli and H.\,Fan
\cite{CF19, CF23, CF24, CF25, F24}.
For the basic theory of conservation laws we refer to \cite{Bbook, HR}.
A general introduction to stochastic processes, 
with various applications, can be found in \cite{CB}.

\section{A stochastic process with random spikes}
\label{sec:2}
\setcounter{equation}{0}

Following \cite{ABS1, ABS2}, weak solutions to the conservation law with discontinuous, gradient-dependent 
flux are defined as follows.

\begin{definition} \label{def:11} A locally integrable function $u=u(t,x)$ defined on $[0,T]\times\R$ 
is a  {\bf weak solution} to 
the Cauchy problem (\ref{fgclaw}), (\ref{id}) if the following holds.
\begi
\item[(i)] The map $t\mapsto u(t,\cdot)$ is continuous with values in $\L^1_{loc}(\R)$, and satisfies (\ref{id}).
\item[(ii)]  There exists a measurable function $\Theta: [0,T]\times\R\mapsto [0,1]$  such that, at every point $(t,x)$
where the partial derivative $u_x(t,x)$ is well defined, we have
\bel{18} \Theta(t,x)~=~\left\{\bega{rl} 1\quad &\hbox{if}\quad u_x(t,x)>-\delta,\\[1mm]
0\quad &\hbox{if}\quad u_x(t,x)<-\delta.\enda\right.\eeq
Moreover, for every test function $\vp\in \C^1_c\bigl(\,]0,T[\,\times\R)$, one has
\bel{weaksol}
 \dint \left\{ u\,\vp_t + \Big[\Theta \,%(u_x) 
 f(u) + \bigl(1-\Theta \, %(u_x) 
 \bigr)g(u)\Big] \vp_x \right\}\, dx\,dt~=~0.\eeq
 \endi
 \end{definition}
Throughout this paper, we restrict the attention to solutions which are 
spatially periodic as in (\ref{uper}), and
where, at each time $t\geq 0$, the function $u(t,\cdot)$ is piecewise $\C^1$ with upward jumps.

\begin{remark} {\rm  As in the case of a single concave flux, under the assumptions {\bf (A1)} the Lax  condition implies that only upward jumps can be admissible.
Indeed, calling  $u^-,u^+$ the left and right states across a jump, we have:
\begi
\item If the flux is $f(u^-)$ on the left and $g(u^+)$ on the right, then
\bel{Lax1}f'(u^-) ~\geq~ {g(u^+) - f(u^-)\over u^+-u^-} ~\geq~g'(u^+)\eeq
implies $u^-<u^+$.

\item If the flux is $g(u^-)$ on the left of the jump, and $f(u^+)$ on the right, then
\bel{Lax2}g'(u^-) ~\geq~ {f(u^+) - g(u^-)\over u^+-u^-} ~\geq~f'(u^+)\eeq
again implies $u^-<u^+$.
\endi
On the other hand, the condition $u^-<u^+$ does not necessarily imply that the inequalities in 
(\ref{Lax1}) or (\ref{Lax2}) should hold.
}
\end{remark}

As shown in \cite{ABS2}, under the assumptions {\bf (A1)} the Cauchy problem (\ref{fgclaw}), (\ref{id}) 
has
infinitely many weak solutions.  In the present paper, for any given sequence $(t_j, X_j)_{j\geq1}$ 
describing the locations of possible new spikes,
we shall determine a unique corresponding solution.  However, this uniqueness will not come from 
an entropy admissibility condition.  Rather, it will be obtained from an assumption on the structure of the solution itself. For this purpose, we introduce the following set of admissible $L$-periodic functions:

\begin{definition}\label{d:22} We denote by $\S$ the set of all {\bf admissible profiles}. This is the family of all 
 functions $w:[0,L]\mapsto \R$ with the following properties.
\begi
\item[{\bf (P1)}] 
$w(\cdot)$ has
finitely many upward jumps, say  at points  $y_k\in [0, L[\,$, ~$k=1,2,\ldots, N$.
Restricted to each open interval $I_k\doteq\,]y_{k-1}, y_k[\,$
between two jumps, 
the function $w(\cdot)$ is $\C^1$ and non-increasing.  
%Either $w(x)= u_0$, or else $w_x(x)<0$.
\item[{\bf (P2)}] 
$w(0)= w(L)$. Moreover 
\bel{pr1} {1\over L}\int_0^L w(x)\, dx\,=\,u_0\,,
\qquad 0\,<\, \inf_x w(x) \,\leq \,\sup_x w(x)\,<\,\alpha.\eeq
\item[{\bf (P3)}] If $w_x(x)\leq -\delta$ at some point $x\in I_k$, then $w_{x}$ is constant on $I_k$.
\endi
\end{definition}

In the following, we consider a strictly increasing sequence of times 
\bel{tseq} 0\,<\, t_1\,<\, t_2\,<~\cdots\qquad\qquad \lim_{j\to\infty} t_j~=~+\infty,\eeq
and a sequence of points $X_j\in [0, L]$.

\begin{theorem}\label{t:21} Assume that the flux functions $f,g$ satisfy {\bf (A1)}. 
Then, for any  sequence of times $t_j$ as in (\ref{tseq})  and points $X_j\in [0,L]$, 
there exists a unique solution $u=u(t,x)$ of (\ref{fgclaw})-(\ref{uper}), (\ref{id}),
which satisfies the inductive rules {\bf (I)-(II)} 
together with the structural condition $u(t,\cdot)\in \S$ for all $t\geq 0$. \end{theorem}

The key step in the proof relies on the following lemma,  describing the formation of a new
$N$-shaped spike (see Fig.~\ref{f:df131}, left).

\begin{lemma}\label{l:21} Assume that the fluxes $f,g$ satisfy  {\bf (A1)}. Let $\omega=\omega(t,x)$ be a $\C^1$ solution of the conservation law (\ref{claw})
with flux $f$, defined for $(t,x)$ in a neighborhood of the origin. Assume $\omega_x(t,x)>-\delta$ and call $\omega_0\doteq\omega(0,0)\in \,]0,\alpha[\,$.

Then there exists $\tau>0$ sufficiently small, and unique shock curves $y'(t)< y''(t)$, $t\in [0, \tau]$, so that 
the following holds.
\begi
\item[(i)] For $t\in \,]0,\tau]$, the piecewise $\C^1$ function
\bel{spk}u(t,x)~=~\left\{ \bega{cl} \omega(t,x)\quad &\hbox{if}\quad x\notin \bigl[y'(t),\, y''(t)\bigr],\\[2mm]
\ds {\alpha\over 2} - {x\over\beta t} \quad &\hbox{if}\quad y'(t)<x< y''(t),\enda\right.\eeq
is a weak solution to the conservation law with two fluxes (\ref{fgclaw}).
\item[(ii)] The initial speeds of the two shocks satisfy
\bel{sspeed}\bega{l}\ds
\lim_{t\to 0+}{d\over dt} y'(t)~=~{g\bigl(w^+(\omega_0)\bigr) - f(\omega_0)\over 
w^+(\omega_0)-\omega_0}~=~g'\bigl(w^+(\omega_0)\bigr) ,\\[4mm]
\ds
 \lim_{t\to 0+}{d\over dt} y''(t)~=~{f(\omega_0) -g\bigl(w^-(\omega_0)\bigr) \over \omega_0-
w^-(\omega_0)}~=~ g'\bigl(w^-(\omega_0)\bigr)   .\enda\eeq
Here $Q^\pm\doteq \Big( w^\pm(\omega_0), g\bigl(w^\pm(\omega_0)\bigr)\Big)$ are the points where the lines through the point $P=\bigl( \omega_0, f(\omega_0)\bigr)$ meet the graph of $g$ tangentially
 (see Fig.~\ref{f:df131}, right, where $u=\omega_0$).
\endi
\end{lemma}

{\bf Proof.}  {\bf 1.} By the assumptions {\bf (A1)},
there exist
unique states $w^\pm=w^\pm (\omega_0)$, with 
\bel{wpm}
w^-\,<\,\omega_0\,<\,w^+,
\eeq
such that the two lines from $P$ to the points $Q^\pm$
are tangent to the graph of $g$. This means (see Fig.~\ref{f:df131}, right, with $u=\omega_0$)
\bel{w+-}
g'(w^+)\,=\,\frac{g(w^+)-f(\omega_0)}{w^+-\omega_0},\qquad
\qquad
g'(w^-)\,=\,\frac{f(\omega_0)-g(w^-)}{\omega_0-w^-}.
\eeq
\v
{\bf 2.} By (\ref{spk}),  for $x\notin \bigl[y'(t),\, y''(t)\bigr]$ the composite function
$u$ is a solution to (\ref{fgclaw}) with flux $f$ and gradient $u_x>-\delta$.
Moreover, for $y'(t)<x< y''(t)$,  the function $u$ satisfies the identity
$$g'\bigl(u(t,x)\bigr) ~=~{\alpha\beta\over 2} -\beta\left( {\alpha\over 2} - {x\over \beta t}\right)~=~{x\over t}\,,$$
hence it is a centered rarefaction fan corresponding to the flux $g$.  For $t>0$ small, its gradient satisfies
$$u_x(t,x)~=\,-{1\over\beta t}~ <\, -\delta.$$
To prove the lemma, it remains to show that the shock curves $y'(t)<y''(t)$ are uniquely determined by the Rankine-Hugoniot conditions.  This fact is not straightforward, because in the  region
between the two shocks the function $u$ at (\ref{spk}) is not Lipschitz continuous when $t\to 0$.
A more careful argument is needed.
\v
{\bf 3.} 
To fix ideas, consider the second shock $y''(\cdot)$ with left and right states
\bel{u-} u^-(t)~\doteq~{\alpha\over 2} - {y''(t) \over \beta t},\qquad
\qquad   u^+(t)~=~\omega\bigl(t, y''(t) \bigr).\eeq
The Rankine-Hugoniot equations yield
\bel{RHfg}
{d\over dt} y''(t)~=~{f\bigl(u^+(t)\bigr) - g\bigl(u^-(t)\bigr)\over u^+(t)- u^-(t)}\,.\eeq
A local solution will be obtained as a fixed point of the Picard operator
\bel{Pic}
(\P z)(t)~=~\int_0^t 
\phi\left(\omega\bigl( s, z(s)\bigr)\,, ~{\alpha\over 2} - {z(s)\over \beta s}\right)\, ds ,
\qquad\mbox{where}\quad
\phi(u^+,u^-) \,\dot=\, \frac{f(u^+)-g(u^-)}{u^+-u^-}. 
\eeq
We will show that, for $t>0$ small, one has
\bel{appr}
u^-(t)\,\approx\, w^-(\omega_0),\qquad\qquad u^+(t)\,\approx\, \omega_0,\eeq
with $w^-<\omega_0$ as in (\ref{wpm}).
Hence the difference $u^+-u^-$ remains uniformly positive, and
 the integrand  $\phi=\phi(u^+, u^-)$ is a $\C^1$ function of its arguments.

Using the second tangency condition in (\ref{w+-}),
we obtain  a sharper bound on the derivative of the shock speed w.r.t.~$u^-$.
Namely
\bel{dumin}\bega{rl}\ds
{\partial\over\partial u^-}\phi(u^+,u^-) &\ds=~\frac{-g'(u^-)(u^+-u^-) + \bigl[ f(u^+)-g(u^-)\bigr] }{(u^+-u^-)^2}\\[4mm]
&=~\O(1)\cdot\Big( |u^+-\omega_0| + | u^- - w^-| \Big).\enda \eeq
Indeed, when $u^+=\omega_0$ and $u^-= w^-$, by (\ref{w+-}) the numerator of 
(\ref{dumin}) vanishes.    The conclusion thus follows from the $\C^1$ regularity of $\phi$.
Here and throughout the sequel, the Landau symbol $\O(1)$ denotes a uniformly bounded quantity.
\v
{\bf 4.}
In order to apply a fixed point argument,  recalling (\ref{w+-}) we fix $\ve_1>0$ sufficiently small and consider the family of Lipschitz functions 
\bel{Fdef}\F~\doteq~\Big\{ z\in W^{1,\infty}\bigl([0,t_0]\bigr);\qquad z(0)=0,
\qquad \bigl|\dot z(t)-g'(w^-)\bigr|\leq \ve_1~~\hbox{for a.e.}~
t\Big\}.\eeq
Here and in the sequel, the upper dot denotes a derivative w.r.t.~time.

We claim that, for $\ve_1, t_0>0$ sufficiently small, the Picard map (\ref{Pic}) maps
$\F$ into itself.

Indeed, since $\phi(\omega_0, w^-)~=~g'(w^-)$, by (\ref{Fdef}) it follows
\bel{er1}
\bigl|\phi(u^+, u^-) - g'(w^-)\bigr|~=~\O(1)\cdot\bigl( |u^+-\omega_0| + |u^--w^-|^2\bigr).\eeq
If now $z(\cdot)\in \F$, then 
\bel{umn}\bigl|u^+(t)-\omega_0\bigr|\,=\,\O(1)\cdot t,\qquad\qquad  |u^-(t)-w^-|\,=\,\O(1)\cdot \ve_1\,.\eeq
Therefore, by first choosing $\ve_1>0$ suitably small and  then a time $t_0>0$ small enough,  for all 
$t\in [0, t_0]$ we obtain
\bel{PZB}\Bigl|\phi\bigl(u^+(t), u^-(t)\bigr) - g'(w^-)\Big|~=~\O(1)\cdot \bigl(t_0 + \ve_1^2)~<~\ve_1\,,\eeq
proving that the Picard iterate satisfies $\P z\in \F$.
\v
{\bf 5.} To achieve a contractive property, 
on $\F$ we consider the distance
$$d(z_1,z_2)~\doteq~\sup_{0<t\leq t_0} \left| \frac{z_1(t)-z_2(t)}{ t}\right|.$$
Let $z_1,z_2$ be two such Lipschitz functions with  $d(z_1,z_2) =\rho$,
so that 
\bel{yz} 
\bigl|z_1(t)-z_2(t)\bigr|~\leq~\rho t\qquad\qquad\forall t\in [0,t_0].
\eeq

To estimate the distance $\Big| (\P z_1)(t)- (\P z_2)(t)\Big|$, for clarity we estimate separately 
the terms arising from the change in $u^+$ and the terms arising from the change in $u^-$. 
We thus obtain
\bel{Pic2}\bega{l}\bigl|
(\P z_1)(t)- (\P z_2)(t)\bigr|\\[4mm]
\ds\qquad =~\int_0^t \bigg|
\phi\left(\omega\bigl( s, z_1(s)\bigr)\,, ~{\alpha\over 2} - {z_1(s)\over \beta s}\right)
-\phi\left(\omega\bigl( s, z_2(s)\bigr)\,, ~{\alpha\over 2} - {z_1(s)\over \beta s}\right)\bigg|
 ds\\[4mm]
 \ds \qquad \quad +\int_0^t \bigg|
\phi\left(\omega\bigl( s, z_2(s)\bigr)\,, ~{\alpha\over 2} - {z_1(s)\over \beta s}\right)
-\phi\left(\omega\bigl( s, z_2(s)\bigr)\,, ~{\alpha\over 2} - {z_2(s)\over \beta s}\right)\bigg|
 ds\\[4mm]
 \qquad \doteq~A(t)+B(t).
 \enda
 \eeq
Since $\omega$ is a $\C^1$ function of its arguments, we have
\bel{Abo}A(t)~=~\O(1)\cdot \int_0^t\bigl| z_1(s)-z_2(s)\bigr|\, ds~=~\O(1)\cdot \int_0^t \rho s\, ds~=~
\O(1) \cdot {\rho t^2\over 2}~\leq~{\rho\over 4}\,t\,,\eeq
provided that $t\in [0,t_0]$ is sufficiently small.

Next,
using the tangency condition (\ref{dumin}) one obtains
\bel{Bbo}\bega{rl}B(t)&\ds=~\O(1)\cdot \int_0^t{\bigl| z_1(s)-z_2(s)\bigr|\over \beta s}\cdot 
\Big(\bigl| u^+(s)- \omega_0\bigr|+ \bigl| u^-(s) - w^-\bigr|\Big)\, ds\\[4mm]
&\ds=~\O(1)\cdot \int_0^t{\bigl| z_1(s)-z_2(s)\bigr|\over \beta s}\cdot \sum_{i=1,2}
\Big(\bigl| u_i^+(s)- \omega_0\bigr|+ \bigl| u_i^-(s) - w^-\bigr|\Big)\, ds\\[4mm]
&\ds \leq ~ \O(1)\cdot \int_0^t {\rho\over\beta} \cdot  \Big[\O(1)\cdot (t_0+\ve_1)\Big]\, ds
~\leq~{\rho\over 4}\, t,\enda
\eeq
provided that $\ve_1, t_0>0$ are chosen  suitably small.

Inserting (\ref{Abo})-(\ref{Bbo}) in (\ref{Pic2}), we conclude that the Picard operator is a strict contraction, hence it admits a unique fixed point $y''(\cdot)\in \F$, defined for $t\in [0,t_0]$.

{}From (\ref{RHfg}), letting $t\to 0+$ one obtains the second limit in (\ref{sspeed}).

An entirely similar argument applies to the first shock curve $y'(\cdot)$, completing the proof.
\endproof

\begin{figure}[ht]
\centerline{\hbox{\includegraphics[width=12cm]{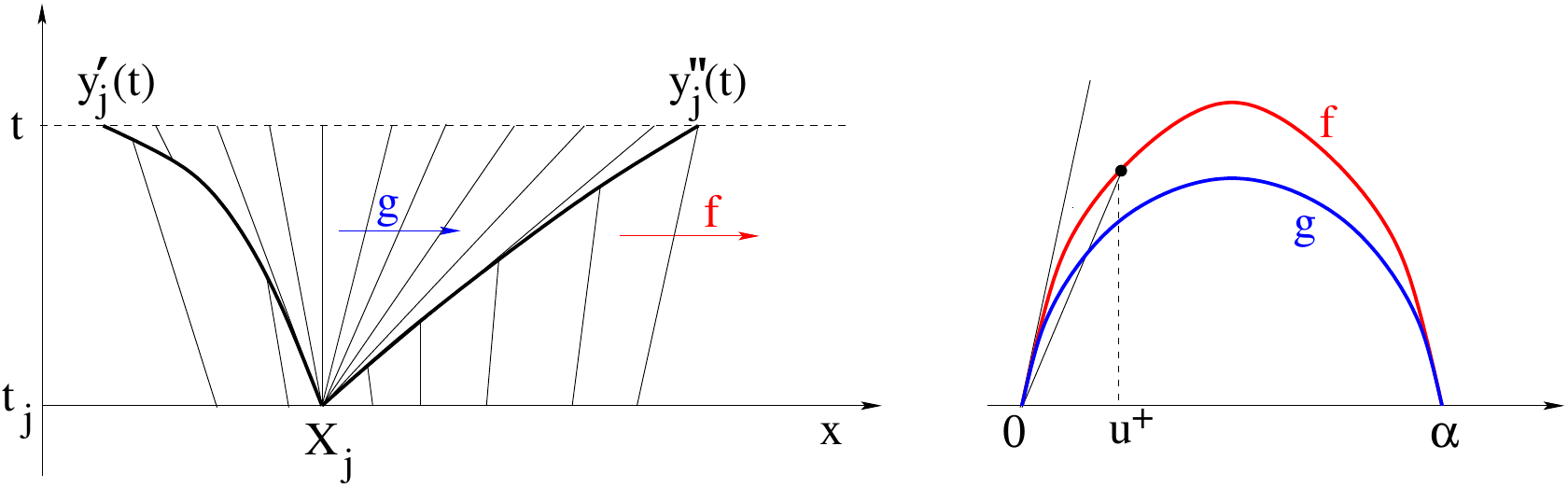}}}
\caption{\small Left: the new shocks $y_j'(t)$, $y_j''(t)$ generated at the  point $(t_j, X_j)$.
Note that the characteristics always flow in toward $y_j''(t)$ from the right. However, they may 
flow outward on the left.   Right: the assumptions $g'(0)=f'(0)$ and $f''(u)<0$ imply 
that $g'(0)> f(u^+)/u^+$ for every $u^+\in \,]0,\alpha[\,$.}
\label{f:df133}
\end{figure}

Relying on Lemma~\ref{l:21} we can now prove the main theorem.
\v
{\bf Proof of Theorem~\ref{t:21}.}

{\bf 1.} During the initial interval $[0, t_1]$, the solution is constant: $u(t,x)=u_0$.

By induction, assume that the solution with the required structural properties has been constructed for $t\in [0, t_j]$. 
We will show that it can be uniquely constructed on the additional interval
\([t_j,t_{j+1}]\). 

\v{\bf 2.} Assume that $(t_j, X_j)$ is a point where $u_x>-\delta$, hence the flux is $f$ in a neighborhood of this point.  Using Lemma~\ref{l:21}, we obtain a unique solution containing a new $N$-shaped spike in a neighborhood of $(t_j, X_j)$. As in (\ref{spk}), this will satisfy
$$u(t,x)~=~{\alpha\over 2} - {x-X_j\over \beta (t-t_j)} \qquad \qquad  y_j'(t)<x<y''_j(t).$$
In particular, on the open interval $\,\bigl]y_j'(t), y''_j(t)\bigr[\,$, the gradient is constant:
\bel{gradc}u_x(t,x)~=~{-1\over\beta( t-t_j)}\,.\eeq
   On the other hand, away from the point $X_j$,
 by genuine nonlinearity the gradient $u_x(t_j,\cdot)$ is uniformly bounded.
Therefore, on some interval $[t_j, t_j']$, a unique piecewise $\C^1$ solution can be constructed by the method of characteristics,
while the positions of the shocks are determined by the Rankine-Hugoniot equations.
\v
{\bf 3.} Next, we
 observe that on the remaining interval $t\in [t'_j, t_{j+1}]$, since no new spikes are generated,  the gradient $u_x$ remains uniformly bounded. We can thus prolong in time the solution, until 
\begi
\item[(i)] two shocks merge into each other, or
\item [(ii)]  the gradient $u_x$ grows past the threshold $-\delta$ on some open interval   $\,\bigl]y_i'(t), y''_i(t)\bigr[\,$.
\endi
Since for $t\in [0, t_{j+1}[\, $ the total number of shocks cannot exceed $ 2j$, each of the above cases can occur only finitely many times.

In case (i), the two upward shocks are simply replaced by a single one, and the construction continues.

In case (ii), the assumption that the flux $g$ is quadratic implies that 
$u_x(t,x)~=~{-1\over \beta(t-t_i)}$ on the interval $\,\bigl]y_i'(t), y''_i(t)\bigr[\,$ containing a centered rarefaction generated at time $t=t_i$.
The gradient crosses the value $-\delta$ simultaneously at all points of this interval, at time
\bel{ti*}t^*_i~=~t_i + \Hat \tau,\qquad\qquad \Hat\tau\,\doteq\,{1\over \beta\delta}\,.\eeq
The solution is thus prolonged by switching the flux from $g$ to $f$.  Notice that, for $t>t^*_i$, this will
change the speed of characteristics over this interval, as well as the speed of the shocks
at the endpoints $y_i'(t)$, $y_i''(t)$.

In a finite number of restarting steps 
we can thus extend the solution on the entire interval $[t_j, t_{j+1}]$.  
\v
{\bf 4.} 
In view of (\ref{tseq}), by induction on $j=1,2,\ldots$ we can  uniquely 
prolong the solution $u=u(t,x)$ for all times $t\in [0, +\infty[\,$, provided that
its values remain inside the open interval $\,]0,\alpha[\,$.
In this last step we check that this is indeed the case.

We recall that the solution $u$ is piecewise $\C^1$ and $u_x\leq 0$ outside the shocks.
Consider a shock with left and right states $u^-(t)<u^+(t)$.   If the flux is $f$ on both sides of the shock,
then characteristics flow inward toward the shock from both sides, and this implies
$${d\over dt} u^-(t)\,>\,0,\qquad\qquad {d\over dt} u^+(t)\,<\,0.$$
The same is true if the flux is $g$ on both sides of the shock.

It remains to check what happens if the flux changes from $f$ to $g$ across the shock.
To fix ideas, consider the left state across a shock 
$$u^-(t)~\doteq~\lim_{x\to y_i''(t)-} u(t,x),$$
assuming that  the flux is $g$ on the left and $f$ on the right.
To make sure that this state does not decrease to zero in finite time, we compute
\bel{du0}\bega{l}\ds{d\over dt} u^-(t)~=~- u_x\bigl(t, y_i''(t)-\bigr) \cdot \left[ g'\bigl(u^-(t)\bigr)-{d\over dt} y_i''(t)\right] \\[4mm]
\qquad\qquad\ds =~ {1\over \beta(t-t_i)}  \cdot \left[ g'\bigl(u^-(t)\bigr)-{f\bigl(u^+(t)\bigr)- g\bigl(u^-(t)\bigr)\over u^+(t)- u^-(t)}\right].\enda
\eeq
If at a first time $\tau>t_i$ we had the convergence $u^-(t)\to 0$ as $t\to \tau-$,  then the right hand side of (\ref{du0}) would converge to 
$$ {1\over \beta(\tau-t_i)}  \cdot \left[ g'\bigl(0)-{f\bigl(u^+(\tau)\bigr)\over u^+(\tau)}\right]~>~0,$$
 (see Fig.~\ref{f:df133}, right). This yields a contradiction.

An entirely similar argument shows that, within the interval $\bigl[ y_i'(t), \, y_i''(t)\bigr]$,  the solution 
always remains strictly smaller than $\alpha$.  This completes the proof.
 \endproof

\begin{remark} {\rm The assumption (\ref{gprop}) that the flux function $g$ is quadratic plays a key role in our model.   Indeed, it implies that, as long as the flux is $g$, the gradient $u_x(t,x)$ is constant in $x$.   As $u_x$ crosses the threshold $-\delta$, the flux function thus switches from $g$ to $f$ simultaneously at all points of an interval.  This greatly simplifies the construction of solutions.

On the other hand, when $u_x>-\delta$ the flux becomes $f$, which is not assumed to be a quadratic polynomial.   Hence the evolution equation for the gradient $u_x$ depends on $u$ as well,
and the solution is no longer piecewise affine (see Figures \ref{f:df104} and \ref{f:df103}).
}
\end{remark}

\begin{figure}[ht]
\centerline{\hbox{\includegraphics[width=11cm]{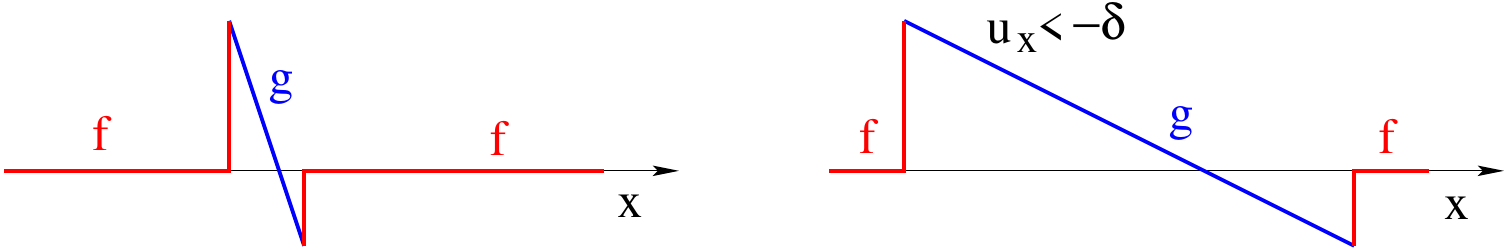}}}
\caption{\small Left: a random spike in the solution to (\ref{fgclaw}).  Right: as time increases, by genuine nonlinearity the gradient $u_x$ increases asymptotically to zero.  }
\label{f:df104}
\end{figure}

\begin{figure}[ht]
\centerline{\hbox{\includegraphics[width=11cm]{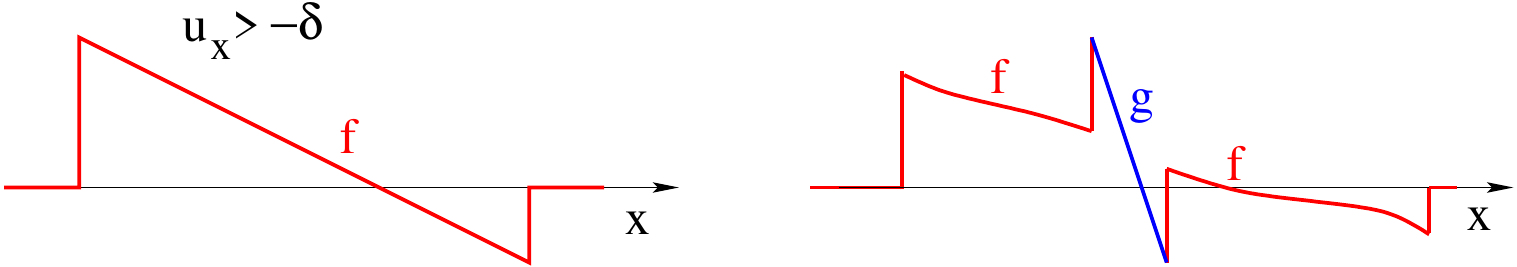}}}
\caption{\small Left: as soon as $u_x>-\delta$, the flux becomes $f(u)$ at every point.
Right:  In this setting, a new random spike can emerge.
Notice that, since the flux $f$ is not quadratic, the gradient $u_x$ is no longer a constant outside the
$N$-shaped spike. }
\label{f:df103}
\end{figure}

Relying on Theorem~\ref{t:21}, we can now assign a probability distribution on the family of all weak 
solutions to the conservation law (\ref{fgclaw}).  
More precisely,  let the fluxes $f,g$ satisfy the assumptions {\bf (A1)}, and let
the constants $\lambda>0$, $u_0\in \,]0,\alpha[\,$ be given. 

Consider the space
$\bfX$  of all sequences $(t_j, X_j)_{j\geq 1}$, where the random times $t_j$ have i.i.d.~increments as in (\ref{P}), and the points $X_j\in [0,L]$ are independent, uniformly distributed random variables.
We take $\bfX$ as our basic probability space.   Note that, with probability 1, the limit in (\ref{tseq}) holds.

For a.e.~sequence of points $(t_j, X_j)_{j\geq 1}$ in $\bfX$, Theorem~\ref{t:21} yields
a unique weak solution $u=u(t,x)$ of the Cauchy problem (\ref{fgclaw}), (\ref{id}), defined for 
all $t\geq 0$ and satisfying the conditions {\bf (I)-(II)} stated in the Introduction.

The push-forward of the probability measure on $\bfX$ through the map
\bel{push}(t_j, X_j)_{j\geq 1}\qquad \longrightarrow\qquad u(\cdot,\cdot)\eeq
defines a probability measure on the family of all weak solutions.
In the remainder of this paper we will study the stochastic process whose random paths are
weak solutions $t\mapsto u(t,\cdot)$ of (\ref{fgclaw}), (\ref{id}), with probability distribution determined by (\ref{push}).

\section{The case $\delta=0$}  \label{sec:3}
\setcounter{equation}{0}

In \cite{ABS2}, only the case where $\delta=0$ in (\ref{fgclaw}) was considered.
However, this model does not lead to an interesting dynamics.  Indeed, with probability 1
as $t\to +\infty$  the solution converges asymptotically to a constant.

\begin{proposition}\label{p:41} Let $f,g$ be $\C^2$ concave functions as in {\bf (A1)},
and assume $\delta=0$.
Consider any constant  initial data  (\ref{id}) with  $\theta\equiv 1$, so that the flux is $f$ at every point.

Then
with probability 
one there holds
\bel{limut} \lim_{t\to +\infty} \bigl\|u(t,\cdot) -u_0\bigr\|_{\L^\infty}~=~0.\eeq
\end{proposition}

{\bf Proof.} ~{\bf 1.} 
As in Fig.~\ref{f:df131}, let $Q^-=(w^-, g(w^-))$ and $Q^+=(w^+, g(w^+))$ 
be the points where a secant line through $P=(u_0, f(u_0))$ is tangent to the graph
of $g$.   Consider the wave speeds
\bel{lpm}
\lambda^+~\doteq~g'(w^+) ~=~{g(w^+) - f(u_0)\over w^+-u_0}~<~\lambda^-~\doteq~g'(w^-) ~=~{ f(u_0)- g(w^-)\over u_0- w^-}\,.\eeq  
Let $t_1$ the first random time where a spike is formed and set
$T\doteq {\frac{L}{ \lambda^--\lambda^+}}$.
We claim that, at time $t_1 + T$, 
the solution is piecewise continuous and the flux is $g$ at all but finitely many points where an upward jump occurs.

Indeed, at any time $t>0$ the solution is piecewise continuous with finitely many upward jumps, say at $y_1(t)< y_2(t)<\cdots<y_n(t)$.   For each open interval
$\,\bigl]y_{i-1}, y_i(t)\bigr[\,$ two cases can occur:
\begi
\item[(i)]  For all $x\in\,\bigl]y_{i-1}(t), y_i(t)\bigr[\,$ one has $u(t,x)= u_0$,
and on this interval the flux is $f$.
\item[(ii)]  For all $x\in\,\bigl]y_{i-1}(t), y_i(t)\bigr[\,$ one has $u_x(t,x)<0$,
and on this interval the flux is $g$.
\endi

Let $D_0(t)\subset [0,L]$ denote the union of all open intervals where
$u(t,x)=u_0$.  If $ \bigl]y_{i-1}, y_i(t)\bigr[\,$ is any one of such intervals,
by (\ref{lpm}) its
endpoints move with speeds
$$\dot y_{i-1}~=~\lambda^-,\qquad\qquad\dot y_{i}~=~\lambda^+.$$
In particular, the length of this interval varies at rate $\lambda^+-\lambda^-$.
As long as $D_0$ is nonempty, its measure thus decreases at rate
$${d\over dt} \hbox{meas}\bigl(D_0(t)\bigr)~\leq~\lambda^+-\lambda^-.$$
In particular, since $\meas\bigl(D_0(t_1)\bigr)=L$, this implies 
$$\meas\left(D_0\Big(t_1+ {L\over \lambda^--\lambda^+}\Big) \right)~=~0,
$$ proving our claim.
\v
{\bf 2.}
By the previous step, for $t>t_1+T$ the function  $u(t,\cdot)$ is a solution to
$$u_t + g(u)_x~=~0,$$
 and no new spikes can be formed.
Since $g''(u)=-\beta$, this implies
\bel{1lip}-{ 1\over \beta   (t- t_1-T)} ~\leq~u_x(t,x)~\leq~0\qquad\quad\hbox{for a.e.}~x\in [0,L].\eeq
In turn, since  $u(t,\cdot)$ is $L$-periodic and can only have upward jumps, this yields the 
 one-sided Lipschitz property
$$u(t,x_2) - u(t, x_1)~\geq~-{x_2-x_1\over \beta   (t- t_1-T)}
\qquad\forall x_1<x_2\,,$$
and hence
\bel{TVt}
\TV\bigl\{ u(t,\cdot)\,;~[0,L[\,\bigr\}~\leq~{2L\over \beta   (t- t_1-T)}\,.\eeq

Furthermore, by the conservation equation, for every $t>0$ one has
$${1\over L} \int_0^L u(t,x)\, dx~=~{1\over L} \int_0^L u_0\, dx~=~u_0\,,$$
and hence
\bel{mm}
\inf_{x\in[0,L]} ~u(t,x)~\leq~u_0~\leq~\sup_{x\in[0,L]} ~u(t,x).\eeq
Together, (\ref{TVt}) and (\ref{mm}) imply
\bel{dist}
\bigl\| u(t,\cdot)-u_0\bigr\|_{\L^\infty}~\leq~{L\over \beta   (t- t_1-T)}\,,\eeq
where the right hand side approaches zero as $t\to +\infty$.
\endproof

\section{Bounds on the total variation}
\label{sec:4}
\setcounter{equation}{0}
Let $u=u(t,x)$ be the solution with constant initial data (\ref{sol0}), corresponding to
 the random sequence $(t_j, X_j)_{j\geq 1}$.
At any time $t>0$, call
\bel{k(t)} k~=~k(t)~\doteq~\max\bigl\{ k\geq 1\,;~~t_k<t\bigr\},\eeq
so that (see Fig.~\ref{f:df112})
\bel{tk}0\,<\,t_1\,<~\cdots ~<\, t_k\,<\,t\,\leq\, t_{k+1}\,.\eeq

\begin{lemma}\label{l:31}
In the above setting, the total variation of the solution $u(t,\cdot)$ over the domain $[0,L]$ is bounded by
\bel{TVu}\TV\bigl\{ u(t,\cdot)\bigr\}~\leq~2\alpha+2\alpha\cdot \min\left\{ n\geq 1\,;~
\sum_{i=1}^n (t- t_{k-i}) \geq {L\over \alpha \beta^*}\right\}.
\eeq
\end{lemma}

{\bf Proof.} {\bf 1.}  As in (\ref{ti*}), set $t_i^*\doteq t_i + {1\over \beta\delta}$. 
For $t>t_i$, consider the set
$V_i(t)\subset [0,L]$ defined as follows.     
\begi
\item If no new spike is generated at time $t=t_i$, then $V_i(t)=\emptyset$ for all $t$.

\item If a new $N$-shaped spike is generated at time $t_i$, for $t_i<t\leq t_i^*$  we let $V_i(t)\doteq \bigl] y'_i(t), y_i''(t)\bigr[\,$
be the open interval bounded by the two shocks originating at $(t_i, X_i)$. 

Moreover, for $t>t_i^*$, we let 
$V_i(t)$ be the set of all points reached by a characteristic starting from a point $y\in V_i(t^*_i)$ at time $t_i^*$. 
\endi
The previous definition requires some explanation. By construction, at time $t_i^*$, on the interval $V_i(t_i^*)$
the flux function switches from $g$ to $f$.   A characteristic of the solution $u$ starting from a point $y\in V_i(t_i^*)$  thus has the form
$$t~\mapsto~ x(t,y)~=~y+ (t-t_i^*) f'\bigl(u(t_i^*,y)\bigr), \qquad\qquad t\in \bigl[t_i^*, ~\tau (y)\bigr].$$
Here $\tau(y)\in \,]t_i^*, +\infty[\,$ is the first time when the characteristic hits a shock, and hence
it terminates.  
For $t> t_i^*$, the above definition can be restated more precisely as
\bel{vi2}V_i(t)~\doteq~\Big\{ x(t,y)\,;~~y\in V_i(t_i^*), ~~\tau(y)>t\Big\}.\eeq
Note that on the remaining set
\bel{V0} [0,L]\setminus \bigcup_{i\geq 1} V_i(t)\eeq
 the solution takes the constant value $u(t,x)=u_0$.
 
 By the property {\bf (P1)} of admissible profiles stated in Definition~\ref{d:22},     the negative part of the variation of $u(t,\cdot)$
is computed by the integral of $u_x$.   We thus have
\bel{TV1}
\TV\bigl\{ u(t,\cdot)\bigr\}~=~2\int_0^L \bigl| u_x(t,x)\bigr|\, dx~=~2\sum_{t_i <t}
\int_{V_i(t)} \bigl| u_x(t,x)\bigr|\, dx.\eeq
  Note that, since $u(t,\cdot)$ takes values in $[0,\alpha]$ and is monotone decreasing on
  $V_i(t)$, one has
\bel{TV2} \int_{V_i(t)} \bigl| u_x(t,x)\bigr|\, dx~\leq~\alpha\eeq
for every $i\geq 1$ and $t> 0$.
\v
{\bf 2.}
By the genuine nonlinearity of both $f$ and $g$, according to (\ref{fgn}) 
along any characteristic the gradient $z(t)= u_x\bigl(t, x(t)\bigr)<0$ decays in time according to
\bel{zdot}{d\over dt} z(t) ~\geq~ \beta^* z^2(t).\eeq
This implies
\bel{TV3} x\in V_i(t)\qquad\implies\qquad 
 {-1\over \beta^*(t-t_i)}~\leq~u_x(t,x)~<~0.\eeq 
 Defining
\bel{TV7}\sigma_i(t)~\doteq~ \int_{V_i(t)} \bigl| u_x(t,x)\bigr|\, dx\,,\eeq
the bound (\ref{TV3}) yields
\bel{TV8}\ell_i(t)~\doteq~\meas\bigl(V_i(t)\bigr)~\geq~\sigma_i(t)\cdot \beta^*(t-t_i).\eeq
\v
{\bf 3.}
The right hand side of (\ref{TV1}) can now be bounded in terms of the optimization problem
\bel{TV9} \hbox{maximize:}\quad\sum_{i=1}^k\sigma_i\eeq
\bel{TV6}\hbox{subject to:}\qquad
 \sum_{i=1}^k \sigma_i\cdot \beta^*(t-t_i)\,\leq\, L, \qquad  \sigma_i \,\leq\,\alpha \quad\hbox{for every}~ i. \eeq
We claim that  this optimization problem (\ref{TV9})-(\ref{TV6}) has a solution of the form
\bel{optsol}\sigma_1=\cdots=\sigma_{p-1}=0, \qquad
\sigma_p\in \,]0, \alpha]\,,\qquad \sigma_{p+1} =\cdots=\sigma_k=\alpha,\eeq
for some index $p\in\{1,\ldots,k\}$.

Indeed, let $(\bar \sigma_1,\ldots, \bar \sigma_k)$ be a maximizer for 
the problem (\ref{TV9})-(\ref{TV6}). By contradiction, assume that there are two indices \(i<j\) such that
\[
\bar \sigma_i>0,\qquad \bar \sigma_j<\alpha.
\]
By decreasing $\bar \sigma_i $ and increasing $\bar \sigma_j$, we can now achieve a higher payoff.   
Indeed, for $\ve>0$ small we can define
$$\sigma_i^\ve ~=~ \bar \sigma_i - {\ve\over t-t_i}\,,\qquad\qquad \sigma_j^\ve ~=~ \bar \sigma_j + {\ve\over t-t_j}\,.$$
 For all $\ve>0$, this construction implies
 $$\sigma_i \beta^* (t-t_i) + \sigma_j^\ve \beta^* (t-t_j)~=~\bar \sigma_i +\bar \sigma_j\,,$$
 $$ \sigma_i^\ve +\sigma_j^\ve ~>~\bar \sigma_i +\bar \sigma_j\,.$$
 Hence, for $\ve>0$, the $k$-tuple obtained by replacing $\bar \sigma_i,\bar \sigma_j$ 
 respectively with
$  \sigma_i^\ve,\sigma_j^\ve $ still satisfies the constraints (\ref{TV6}) and yields a higher
payoff in (\ref{TV9}).
This contradiction shows that the maximizer has 
the form (\ref{optsol}).
\v
{\bf 4.} Using (\ref{TV1})
by the previous step we obtain
\bel{TV10}
\TV\{u(t,\cdot)\}
~=~
2\sum_{t_i<t}\sigma_i(t)
~\leq~
2(k+1-p)\alpha.
\eeq

Note that
\bel{TV11}
k+1-p
~\leq~
1+\min\left\{
n\geq 1\,;~
\sum_{i=1}^{n}(t-t_{k-i})\geq {L\over \alpha\beta^*}
\right\}.
\eeq
Together, (\ref{TV10})-(\ref{TV11}) yield 
(\ref{TVu}).
\endproof

\begin{figure}[ht]
\centerline{\hbox{\includegraphics[width=9cm]{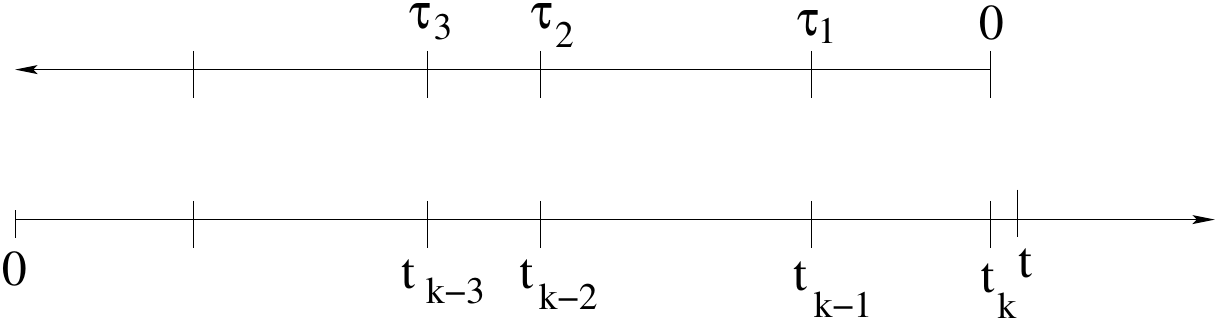}}}
\caption{\small  The times $t_j$ where a new spike can be generated, and the
corresponding times $\tau_i$ defined by $\tau_i \doteq t_k - t_{k-i}$ }
\label{f:df112}
\end{figure}

\begin{figure}[ht]
\centerline{\hbox{\includegraphics[width=12cm]{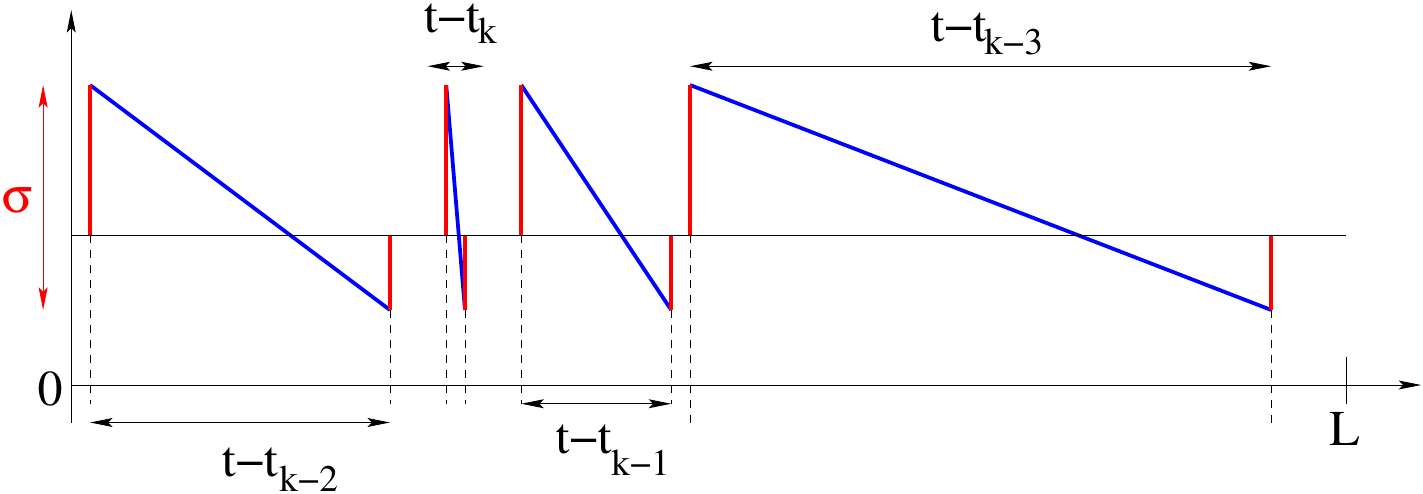}}}
\caption{\small  Sketch of the estimate in (\ref{Ndef}).   Inside the interval $[0,L]$, we 
fit as many spikes as possible.  The spike generated at time $t_{k-i}$ has amplitude $\sigma_{k-i}$ and is supported on an interval of size $\approx \sigma_{k-i} \beta(t-t_{k-i})\geq \sigma_{k-i} \beta\tau_i$.}
\label{f:df113}
\end{figure}

For a fixed time $t>0$, let $k=k(t)$ be as in (\ref{k(t)}).
Thanks to  Lemma~\ref{l:31}, a bound on the expected value of the total variation 
of $u(t,\cdot)$
can now be achieved
by estimating the expected value of the discrete random variable
\bel{Ndef}  N~\doteq~ \min\left\{ n\geq 1\,;~
\sum_{i=1}^n (t_k- t_{k-i}) \geq {L\over \alpha \beta^*}\right\}~\geq~
\min\left\{ n\geq 1\,;~
\sum_{i=1}^n (t- t_{k-i}) \geq {L\over \alpha \beta^*}\right\}.\eeq
Setting
$$\Tilde\theta_j~\doteq~t_{k-j+1}- t_{k-j}\,,$$
we reach the equivalent expression
\bel{Ndef2} N~\doteq~ \min\left\{ n\geq 1\,;~
\sum_{j=1}^n j\,\Tilde\theta_j \geq {L\over \alpha \beta^*}\right\},\eeq
where now the $\Tilde\theta_j$ are independent, identically distributed random variables, 
with
$$\hbox{Prob.}\{ \Tilde\theta_j>s\}~=~e^{-\lambda s}.$$
The further change of variable
$$\theta_j~\doteq~{\alpha \beta^*\over L }\Tilde\theta_j$$
leads to the simpler formula
\bel{Ndef3}  N~=~ \min\left\{ n\geq 1\,;~
\sum_{j=1}^n j\,\theta_j \geq 1\right\},\eeq
where $(\theta_j)_{j\geq 1}$ is a sequence of independent, identically
distributed random variables with 
\bel{P2}\hbox{Prob.}\{ \theta_k > s\}~=~e^{-\mu s},\qquad\qquad \mu= {\lambda L\over\alpha \beta^*}\,.\eeq

\begin{proposition}\label{prop:41}
In the above setting, the expected total variation of $u(t,\cdot)$ is bounded by
\bel{TVineq}\ov W(t)~\doteq~E\Big[ \TV\bigl\{ u(t,\cdot)\bigr\}\Big]~
\leq~2\alpha \cdot \left( 1+ e^{\lambda L/\alpha\beta^*}\right).\eeq
\end{proposition}

{\bf Proof.} {\bf 1.} In view of Lemma~\ref{l:31}, we need to estimate the expected value of the random variable $N$ in (\ref{Ndef3}).
By definition, 
\bel{EN1} \bega{rl}E[N]&=~\ds \sum_{k\geq 1} k \cdot\prob\{N=k\}~=~\sum_{k\geq 1} 
k\cdot \Big(\prob\{N\geq k\}-\prob\{ N\geq k+1\}\Big)\\[4mm]
&=~\ds\sum_{k\geq 1} 
k\cdot \prob\{N\geq k\}-\sum_{k\geq 1}(k-1)\cdot \prob\{ N\geq k\}~=~\sum_{k\geq 1} 
\prob\{N\geq k\}.
\enda
\eeq
We now have
\bel{EN2}\prob\{N\geq k\}~=~\prob\left\{ \sum_{j=1}^{k-1}j\theta_j <1\right\}~\leq~
\prob\left\{ \sum_{j=1}^{k-1}\theta_j <1\right\}.\eeq
Since the $\theta_j$ are independent random variables with exponential distribution
(\ref{P2}), their sum can be estimated by
\bel{EN3}\prob\left\{ \sum_{j=1}^{k-1}\theta_j <1\right\}~=~
 {\mu^{k-1}\over \Gamma(k-1)} \int_0^1 
x^{k-2} e^{-\mu x}\, dx~\leq~~{\mu^{k-1}\over (k-2)!} \cdot{1\over k-1}
~=~{\mu^{k-1}\over (k-1)!} \,,\eeq
where $\mu$ is the constant in (\ref{P2}).
Combining (\ref{EN1})--(\ref{EN3}) we obtain
\bel{EZ4} E[N]~\leq~\sum_{k\geq 1} {\mu^{k-1}\over (k-1) !} ~=~e^{\mu}.\eeq
 Hence
\bel{TV5}E\Big[ \TV\bigl\{ u(t,\cdot)\bigr\}\Big]~
\leq~2\alpha\cdot\bigl( 1+  E\bigl[  N\bigr]\bigr)~
\leq~2\alpha\cdot( 1+  e^\mu).\eeq
Recalling (\ref{P2}) one obtains (\ref{TVineq}).
\endproof

\subsection{A bound on the variance 
of the total variation.}
With the notation introduced at (\ref{TVineq}), the variance of the total variation is
\bel{Var1} \Var\Big(\TV \bigl\{u(t,\cdot)\bigr\}\Big) ~=~ 
E\Big[\Big(TV\bigl\{u(t,\cdot)\bigr\}- \ov W(t)\Big)^2\Big].\eeq
Since the expected value of the total variation trivially satisfies $\ov W(t)\geq 0$,
we have
\bel{Var2}
E\Big[\Big(TV\bigl\{u(t,\cdot)\bigr\}- \ov W(t)\Big)^2\Big]~
\leq~E\Big[\Big(TV\bigl\{u(t,\cdot)\bigr\}\Big)^2\Big]~
\leq~
4\alpha^2E[(1+N)^2].\eeq

Observing that
\[
N^2~=~\sum_{k=1}^N (2k-1)~=~
\sum_{k\geq 1}(2k-1)\mathbf 1_{\{k\leq N\}},
\]
taking expectations we obtain
\bel{EN5}
E[N^2]
~=~
\sum_{k\geq 1}(2k-1) \prob\{N\geq k\} ~\leq~\sum_{k\geq 1}(2k-1)\frac{\mu^{k-1}}{(k-1)!}\,.
\eeq
Using the variable $j=k-1$ one obtains
\bel{EN6}
E[N^2]
~\leq~
2\sum_{j\geq 0}j\frac{\mu^j}{j!}
+
\sum_{j\geq 0}\frac{\mu^j}{j!}~=~ (2\mu+1)e^\mu.
\eeq
Therefore
\bel{EN7}
E[(1+N)^2]
~=~
E[N^2]+2E[N]+1
~\leq~
(2\mu+1)e^{\mu}+2e^{\mu}+1
~=~
(2\mu+3)e^{\mu}+1.
\eeq
By the definition of $\mu$ at (\ref{P2}), this yields
\bel{Var}
\Var\bigl(\TV\{u(t,\cdot)\}\bigr)
~\leq~
4\alpha^2\left[\left( 2\frac{\lambda L}{\alpha \beta^*}+3\right)e^{{\lambda L/\alpha \beta^*}}+1\right].
\eeq

\section{Bounds on the average velocity and  acceleration}
\label{sec:5}
\setcounter{equation}{0}
Since both fluxes $f,g$ are concave down, an upper bound on the average velocity of cars can be achieved by an application of Jensen's inequality.  This bound is deterministic, i.e., it
does not depend
on the random sequence $(t_j, X_j)_{j\geq 1}$ that determines a particular solution.

\begin{proposition}\label{p:31} At every time $t\geq 0$, for every random solution $u(t,\cdot)$ 
the average velocity satisfies the bound 
\bel{Vavb}V^{avg}(t)~\leq~{f(u_0)\over u_0}\,.\eeq
\end{proposition}
{\bf Proof.} We recall that $f\geq g$ and that the total number of cars within $[0,L]$ 
remains constantly equal to $u_0 L$.  Using
Jensen's inequality we thus obtain
\bel{vav}\bega{rl}
V^{avg}(t)&\ds \doteq~ {1\over u_0 L}\left[ \int_{\theta=0} g\bigl(u(t,x)\bigr)\, dx+ \int_{\theta=1}f \bigl(u(t,x)\bigr)\, dx\right]\\[4mm]
&\ds
\le~{1\over u_0}\cdot 
\frac{1}{L}\int_0^L f\bigl(u(t,x)\bigr)\,dx ~
\le~{1\over u_0}
f\!\left(\frac1L\int_0^L u(t,x)\,dx\right)~=~{1\over u_0} f(u_0).\enda
\eeq
\endproof

Next, we derive a bound on the expected average acceleration.
In a region where the flux is 
$f(u)$ (or $g(u)$, respectively), the speed of cars is
$$v^f(u)~=~{f(u)\over u}\qquad\hbox{or}\qquad v^g(u)\,=\,{g(u)\over u}\,.$$
As a preliminary, we show that this speed decreases across every admissible (i.e., upward) jump, where $0<u^-<u^+<\alpha$. Four cases need to be considered, depending on whether the flux to the left and to the right is $f$ or $g$.

Since $f(0)=g(0)=0$ and $f,g$ are concave down,  for a jump with the same flux to the left and to the right  one immediately gets the inequalities
\bel{spe1} {f(u^-)\over u^-}~>~{f(u^+)\over u^+},\qquad  {g(u^-)\over u^-}~>~{g(u^+)\over u^+}\,.\eeq
For a jump where the flux is $f$ on the left and $g$ on the right, we have
$$ {f(u^-)\over u^-}~ >~{f(u^+) \over u^+}~>~{g(u^+)\over u^+}\,.$$
Finally, for a jump where the flux is $g$ on the left and $f$ on the right, the strict 
concavity of $g$ and the admissibility  of the shock imply
\bel{sadm}{g(u^-)\over u^-}~>~g'(u^-) ~\geq~{f(u^+)-g(u^-)\over u^+-u^-} \,.\eeq
Using (\ref{sadm}) we obtain
\bel{ses2} f(u^+)~=~{g(u^-)\over u^-}\cdot u^- + {f(u^+)-g(u^-)\over u^+-u^-}\cdot (u^+-u^-) ~ < ~{g(u^-)\over u^-} \bigl[ u^-+(u^+-u^-)\bigr], \eeq
hence
$${f(u^+)\over u^+}~<~{g(u^-)\over u^-}\,,$$
as claimed.
\v
 Next, consider a region where the flux is $f$. Here and in the sequel, first and second derivatives w.r.t.~time will be denoted by one and two upper dots, respectively.
Along a car trajectory $t\mapsto x(t)$, with 
$$\dot x(t)~=~v^f \bigl(u\bigl(t, x(t))\bigr)~=~  {f\bigl(u(t,x(t))\bigr)\over u\bigl(t, x(t)\bigr)}, $$ the acceleration  is computed by
$$\ddot x(t)~=~{d\over dt} \left({f(u(t,x(t)))\over u(t, x(t))}\right)~=~{f'(u) [u_t + \dot x u_x] u -
f(u) [u_t + \dot x u_x]\over u^2}\,.$$
Since $u_t = - f'(u) u_x$, we obtain
$$\ddot x~=~{f'(u) u - f(u)\over u^2}\cdot [\dot x - f'(u)] u_x~=~
{1\over u} \left( {f(u)\over u } - f'(u)\right)^2 \cdot (- u_x)\,.$$
An entirely similar formula holds when the flux is $g$. 
Recalling that  at every time $t\geq 0$ the total amount of cars over the interval $[0,L]$ is 
$u_0 L$,
the average acceleration is computed by
\bel{acac}A^{avg} (t)~=~-{1\over  u_0 L} \left[  \int_{\theta=1} \left({f(u)\over u} - f'(u)\right)^2  u_x\, dx+ \int_{\theta=0} \left( {g(u)\over u} - g'(u)\right)^2 u_x\, dx\right].\eeq
It is understood that all integrals range within one spatial period $[0,L]$.  
We can now prove
\begin{proposition}\label{p:32} In the above setting, the expected value of the average acceleration
satisfies the bound
\bel{ExAc}
E\bigl[A^{avg}(t)\bigr]~
\le~
\frac{\alpha^3\beta^2}{u_0 L}
\left(1+e^{\lambda L/\alpha\beta^*}\right).
\eeq
\end{proposition}

{\bf Proof.} Setting 
$$
M\,\doteq\,
\max_{\psi\in\{f, g\}} ~\sup_{u\in[0,\alpha]}
\left(
\frac{\psi(u)}{u}-\psi'(u)
\right)^2,$$
we have
\bel{acbd}
A^{avg}(t)
~\le~
\frac{M}{u_0 L}
\int_0^L |u_x|\,dx ~\le~ \frac{M}{2u_0 L}TV\bigl\{u(t,\cdot)\bigr\}.
\eeq
An upper bound on $M$ can be obtained observing that, by the assumptions {\bf (A1)}, 
for every $u\in [0,\alpha]$ one has
$$0\leq {f(u)\over u}\leq f'(0) = {\alpha\beta\over 2},\qquad\qquad 0\leq {g(u)\over u}\leq g'(0) = {\alpha\beta\over 2}\,,$$
$$
f'(u)\geq f'(\alpha) =- {\alpha\beta\over 2}\,,\qquad g'(u)\geq g'(\alpha)=- {\alpha\beta\over 2}\,.$$
Hence $M\leq (\alpha\beta)^2$.
By (\ref{acbd}) this yields
\bel{AvAc}
A^{avg}(t)~
\le~ \frac{(\alpha\beta)^2}{2u_0 L}TV\bigl\{u(t,\cdot)\bigr\}.
\eeq

Taking expectations and recalling (\ref{P2}), by (\ref{TV5})  we obtain (\ref{ExAc}).
\endproof

\section{A 2-dimensional set of profiles}
\label{sec:6}
\setcounter{equation}{0}

Let $\bar u\in \S$ be an initial data
within the set of admissible profiles introduced in Definition~\ref{d:22}.
Aim of this section is to prove that, with positive probability, 
after some fixed time $T$ the solution $u(T,\cdot)$ contains one single upward jump, and
 lies within a family of functions depending only on 2 parameters.

\begin{lemma} \label{l:61} There exists constants $T_1, T_2, T_3>0$, with $T_2=1$, such that the following holds. 
Let $u=u(t,x)$ be a solution of (\ref{fgclaw}) with initial data $u(0,\cdot)=\bar u\in \S$. Assume that:

\begi
\item During the time interval $[0, ~T_1]$ no new spike is generated.
\item During the time interval $[T_1,~ T_1+T_2]$ exactly one new  spike is generated.
\item During the time interval $[T_1+T_2, ~T_1+T_2+T_3]$ no new spike is generated.
\endi
Then (see Figures~\ref{f:df116} and \ref{f:df117})  at the terminal time $T=T_1+T_2+T_3$, the function $u(T,\cdot)$ contains one single
upward jump.  

Moreover, the set of all profiles $u(T,\cdot)$ that can be obtained in this way is a 2-dimensional manifold.
\end{lemma}

\begin{figure}[ht]
\centerline{\hbox{\includegraphics[width=13cm]{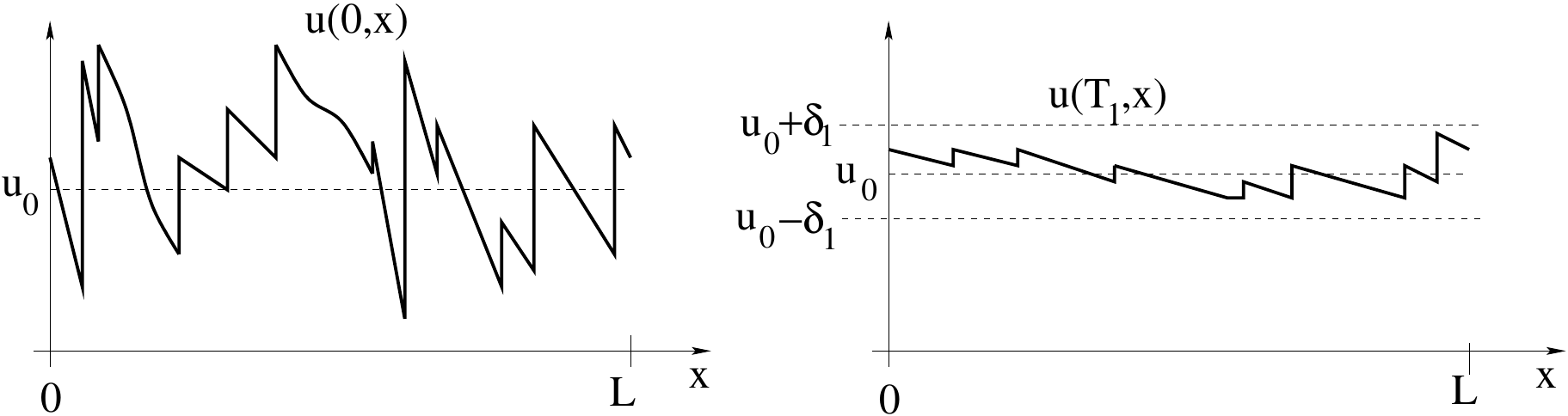}}}
\caption{\small  Left: an initial condition $u(0,\cdot)\in \S$.  Right: at time $T_1$, the negative gradient has decayed enough so that $-\delta< u_x<0$ at a.e.~$x\in [0,L]$.   Hence the flux is $f$ and spikes can be generated anywhere.  }
\label{f:df116}
\end{figure}

\begin{figure}[ht]
\centerline{\hbox{\includegraphics[width=13cm]{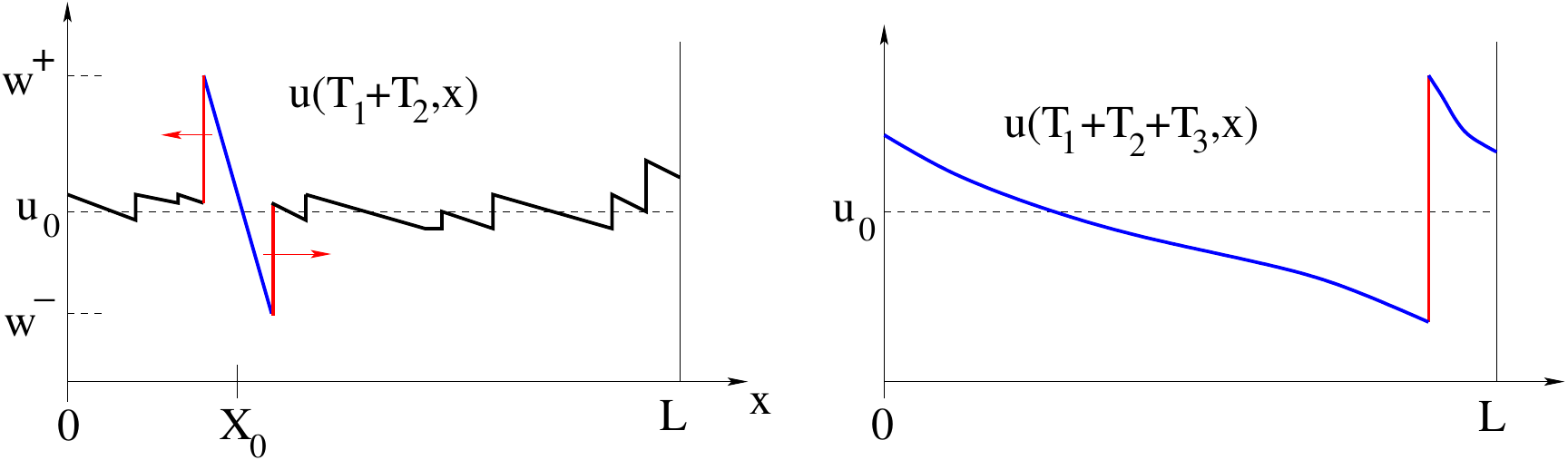}}}
\caption{\small  Left: a new spike is generated at time $t_0\in [T_1, T_1+T_2]$, at a random point $X_0$.  For $t>t_0$, the spike keeps expanding.   Right: after enough time,  the two large upward jumps join together (after one of them crosses the boundary and re-enters through the opposite side),
and the solution contains a single shock.}
\label{f:df117}
\end{figure}

{\bf Proof.}
{\bf 1 (Choosing $T_1$ large enough, $u(T_1,\cdot)$ is nearly constant).}
Fix a  constant $\delta_1>0$ small enough, whose precise value will be determined later.
Then choose
\bel{T1def}
T_1~>~
\max\left\{
\frac1{\beta^*\delta}\,,~
\frac{L}{\beta^*\delta_1}
\right\}.
\eeq
Observe that, if no new spike is generated during $[0,T_1]$, then by (\ref{fgn})
the genuine nonlinearity of the fluxes $f,g$ yields
$$-{1\over\beta^* t}~\leq~u_x(t,x)~\leq~0,\qquad\qquad\forall t\in [0, T_1] ~\hbox{and a.e.~}x\in [0,L].$$
In particular,  at time $t=T_1$ the choice (\ref{T1def}) yields
\bel{T11} -\delta\,<\,u_x(T_1,x)\,\leq\,0,\qquad u_x(T_1,x)~>\, -{\delta_1\over L}
\qquad\hbox{for a.e.~}x\in [0,L].\eeq
In turn, this implies
\bel{T12}\TV\bigl\{ u(T_1,\cdot)\bigr\} ~<~2\delta_1\,,\qquad\qquad 
\bigl| u(T_1,x)- u_0\bigr|\,<\, \delta_1\qquad\forall x\in [0,L]\,.\eeq
We have use here the fact that $u(t,\cdot)$ is $L$-periodic and its average value is $u_0$,  for every $t\geq 0$.
\v
{\bf 2.} Next, we choose $T_2=1$ and let $t_0\in [T_1, ~T_1+T_2]$ the time when the new spike is generated, say at a point $X_0\in [0, L]$.   Call $\bar u_0 = u(t_0, X_0)$.
By (\ref{T12}) it follows
\bel{T15}\bigl| u(t,x)- u_0\bigr|< \delta_1\qquad\forall x\in [0,L]\,,~~T_1\leq t\leq t_0\,,\eeq
and hence
\bel{T13} |\bar u_0-u_0\bigr|~<~\delta_1\,.\eeq

With reference to Fig.~\ref{f:df136}, 
call  $y'(t)<  y''(t)$ the positions of the two new upward shocks, and 
denote by
$$u^-(t)~\doteq~u(t, y''(t) -)  ~<~ u^+(t) ~\doteq~u\bigl(t, y'(t) +\bigr)$$ 
the states next to these two shocks.
Recalling (\ref{gprop}), the gradient of the solution is
$$u_x(t,x)~=~{-1\over \beta (t-t_0)}~\leq~-\delta\,,\qquad\forall t\in [t_0, ~t_0+ \Hat \tau],~~
x\in \bigl]y'(t),\, y''(t)\bigr[\,.$$
As in (\ref{ti*}), here $\Hat \tau=1/\beta\delta$ is the time it takes for the gradient to decay from $-\infty$ up to $-\delta$.
Note that, by $L$-periodicity, a shock that crosses one of the boundaries re-enters through the opposite boundary.  

Setting $\Hat \tau_0\doteq t_0+\Hat \tau$, two cases will be considered.

{\bf CASE 1} (see Fig.~\ref{f:df136}, right): The two shocks at $y',y''$ merge into a single shock at a time $T^*\leq \Hat\tau_0$.  

In this case, at $t=T^*$ the function $u(T^*,\cdot)$ is piecewise affine, with a single jump
at the point $y'(T^*)=y''(T^*)$.  For $t\in [T^*, \Hat\tau_0]$, the function $u$ provides a solution to
the scalar conservation law $u_t+g(u)_x=0$, while  for $t\in [\Hat\tau_0, T]$, the function $u$ is a solution to $u_t+f(u)_x=0$.   In both cases, $u(t,\cdot)$ is piecewise $\C^1$, with one single shock
at a point $y(t)\in [0,L[\,$, and $u_x(t,x)<0$ for $x\not= y(t)$.

{\bf CASE 2} (see Fig.~\ref{f:df136}, left): The two shocks at $y',y''$ remain separated for all $t_0< t \leq  \Hat\tau_0$.   This requires a more careful analysis.  We claim that
\begi
 \item[(1)] Inside the interval $\bigl[ y'(\Hat\tau_0),\, y''(\Hat\tau_0)\bigr]$ the function 
 $u(\Hat \tau_0,\cdot)$ is affine, while outside it satisfies $\bigl|u(\Hat\tau_0,x ) - u_0\bigr|\leq\delta_1$.
\item[(2)] The function $u(\Hat \tau_0,\cdot)$  has  two large jumps
at the points $y'(\Hat\tau_0)$, $y''(\Hat\tau_0)$, with
$$u\bigl(\Hat\tau_0,\, y'(\Hat\tau_0)+\bigr)-w^+(u_0)~\approx~0,\qquad\qquad
 u\bigl(\Hat\tau_0,\, y''(\Hat\tau_0)-\bigr)-w^-(u_0)~\approx~0.$$
 In fact,  the two above differences can be rendered arbitrarily small by choosing $\delta_1>0$  small enough.
\endi
Observing that for $t>\Hat\tau_0$ the function $u$ provides a solution to the scalar conservation law
\bel{fclaw} u_t+f(u)_x~=~0,\eeq choosing $T$ large enough, we will prove that $u(T,\cdot)$ contains a single shock.

\begin{figure}[ht]
\centerline{\hbox{\includegraphics[width=13cm]{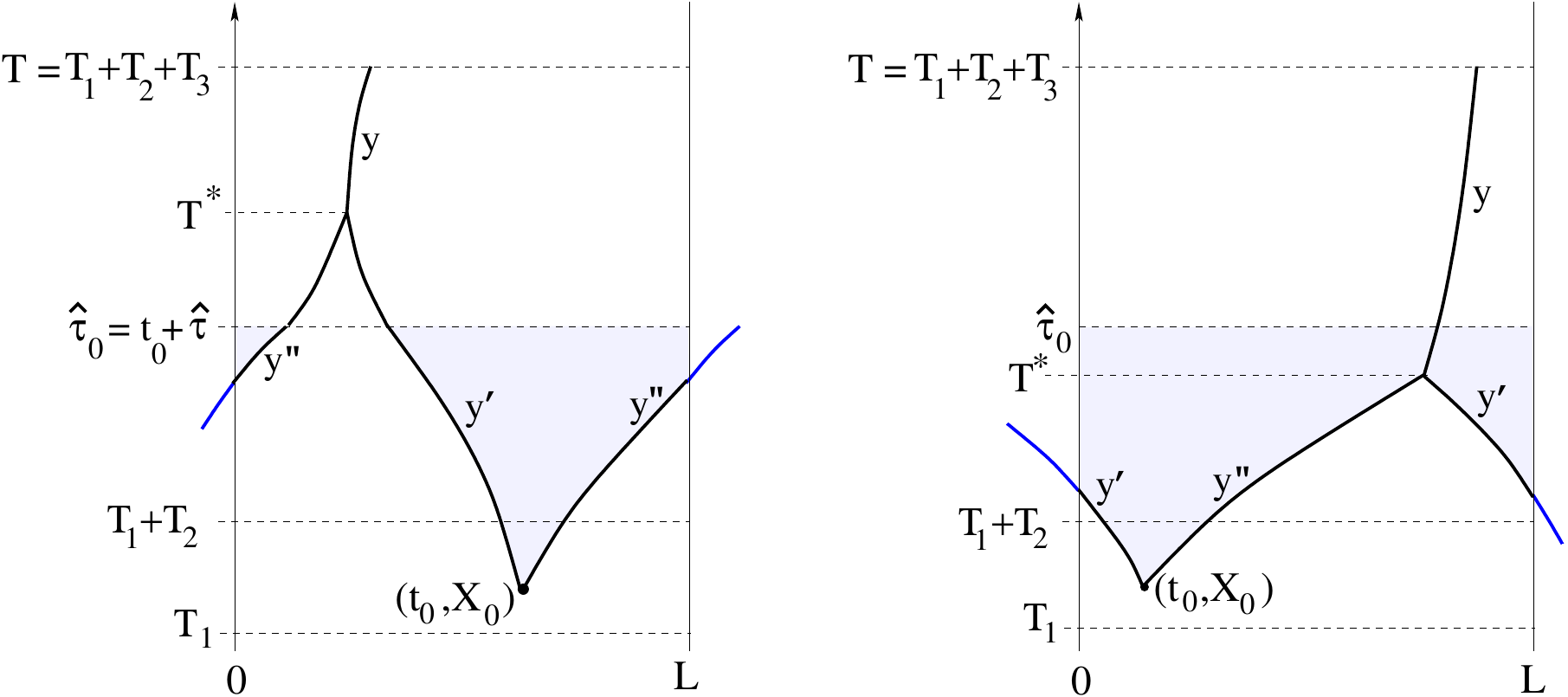}}}
\caption{\small As a new spike is generated at $(t_0, X_0)$, the flux function becomes $g(u)$ in the interval between the two shocks $y'(t)$ and $y''(t)$ (the shaded region).  Note that the flux switches back to $f$ at time $t_0+\Hat \tau$.  
The two figures show the two cases where the shocks merge into each other respectively 
after or before the time
$t_0+\Hat\tau$.  It is convenient to think of $u(t,\cdot)$ as prolonged outside $[0,L]$ by periodicity, so that a shock that crosses one of the boundaries re-enters the domain through the other boundary.}
\label{f:df136}
\end{figure}
%
%{\color{blue}
%To motivate the next step of the proof, we observe that the solution to (\ref{fgclaw}), with constant initial data
%$U(t_0,x)=u_0$ and with a spike generated at the point $(t_0, X_0)$, has the explicit form
%\bel{Us1}U(t,x)~=~\left\{ \bega{cl}u_0\quad &\hbox{if}\quad x\notin \bigl[y'(t),\, y''(t)\bigr],\\[2mm]
%\ds {\alpha\over 2} - {x-X_0\over\beta (t-t_0)} \quad &\hbox{if}\quad y'(t)<x< y''(t),\enda\right.\eeq
%\bel{Us2}y'(t)\,=\, X_0 + (t-t_0) g'\bigl( w^+(u_0)\bigr),\qquad\qquad y''(t)\,=\, X_0 + (t-t_0) g'\bigl( w^-(u_0)\bigr).\eeq
%As in Lemma~\ref{l:21}, here $Q^\pm\doteq \Big( w^\pm(u_0), f\bigl(w^\pm(u_0)\bigr)\Big)$ are the points where the lines through the point $P=\bigl( u_0, f(u_0)\bigr)$ meet the graph of $g$ tangentially
% (see Fig.~\ref{f:df131}, right).
%In particular, for this special solution the sizes of the the shocks remain constant:
%$$U^-(t)~\doteq~U(t, y''(t) -) ~= ~ w^-(u_0)~<~w^+(u_0)~=~ U^+(t) ~\doteq~U\bigl(t, y'(t) +\bigr).$$ 
%
%}
%
%

\begin{figure}[ht]
\centerline{\hbox{\includegraphics[width=10cm]{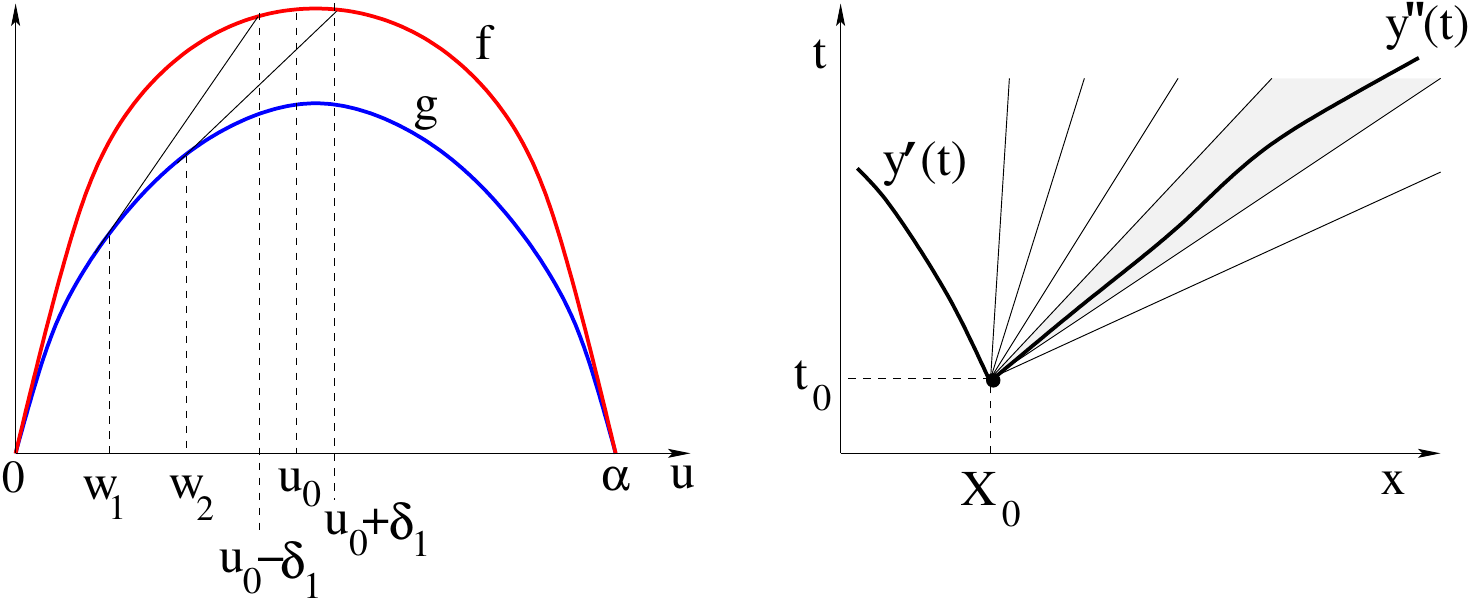}}}
\caption{\small  Left: the interval $[w_1, w_2]$ considered at (\ref{gp2}).
Right: the shock curve $y''(\cdot)$ remains inside the shaded region $\Omega$ defined at (\ref{Om}).}
\label{f:df134}
\end{figure}

\v
{\bf 3.} 
For  $t\in \,]t_0, \Hat \tau_0]$, the solution $u$ satisfies the relations
\bel{upr1}\left\{\bega{cl}\ds
g'\bigl(u(t,x)\bigr) ~=~{x-X_0\over t-t_0}\qquad &\hbox{if}\quad y'(t)<x<y''(t),\\[4mm]
\bigl| u(t,x) - u_0\bigr|~\leq~\delta_1\qquad &\hbox{otherwise.}\enda
\right.
\eeq
As shown in Fig.~\ref{f:df134}, left, consider the points on the graph of $f$ corresponding to $u=u_0\pm \delta_1$.   Call $w_1<w_2$ the values of $u$ at which the lines through
$\bigl( u_0\pm \delta_1,\, f(u_0\pm \delta_1)\bigr)$ touch the graph of $g$ tangentially.  Note that this implies
\bel{gp2}g'(u_0), g'(\bar u_0)~\in ~\bigl] g'(w_2), \, g'(w_1)\bigr[\,.\eeq

We claim that, for $t\in \,]t_0, \Hat \tau_0]$, the shock $y''(t)$ remains in the sector
\bel{Om} \Omega~\doteq~\left\{ (t,x)\,;~g'(w_2)\,<\,{x-X_0\over t-t_0} \,<\, g'(w_1)\right\},\eeq
 see Fig.~\ref{f:df134}, right.

Indeed, as in the proof of Lemma~\ref{t:21}, 
$$\lim_{t\to t_0+} {d\over dt} y''(t)~=~g'(\bar u_0)~\in ~\bigl] g'(w_2), \, g'(w_1)\bigr[\,.$$
Hence, our claim certainly holds for $[t_0, \, t_0+\epsilon[\,$ with $\epsilon>0$ small enough.
To complete the proof, we show that, if at some first time $t^*$ the shock $y''(\cdot)$ ever touched one of the boundaries of $\Omega$, it would be immediately pushed inward.

Assume by contradiction that at some first time $t^*>t_0$ we have 
$${y''(t^*) - X_0\over t-t_0}~=~g'(w_2),\qquad\hbox{hence}\quad u^-(t^*)= w_2\,.$$
Then the Rankine-Hugoniot speed of the shock is 
$${d\over dt} y''(t^*)~=~{f\bigl(u(t^*, y''(t^*)+)\bigr) - g(w_2)\over 
u\bigl(t^*, y''(t^*)+\bigr) - w_2}~>~{f(u_0+\delta_1) - g(w_2)\over 
u_0 +\delta_1 - w_2}~>~g'(w_2),$$
yielding a contradiction.

Similarly, assume that at some first time $t^*$ we have
$${y''(t^*) - X_0\over t-t_0}~=~g'(w_1),\qquad\hbox{hence}\quad u^-(t^*)= w_1\,.$$
Then the Rankine-Hugoniot speed is 
$${d\over dt} y''(t^*)~=~{f\bigl(u(t^*, y''(t^*))\bigr) - g(w_1)\over 
u(t^*, y''(t^*)) - w_2}~<~{f(u_0-\delta_1) - g(w_1)\over 
u_0 -\delta_1 - w_2}~<~g'(w_1),$$
which again yields a contradiction.

By choosing $\delta-1>0$ small enough, we can thus assume that 
the left state across the shock $y''(t)$ satisfies
$$u^-(t)~=~u\bigl(t, y''(t)-\bigr)~\in~]w_1,w_2[\,,\qquad \forall~t\in \,]t_0, \Hat \tau_0], $$
where $]w_1,w_2[$ is an arbitrarily small neighborhood of $w^-(u_0)$.

By a similar argument, we can assume that 
the right state across the shock $y'(t)$ satisfies
$$u^+(t)~=~u\bigl(t, y'(t)+\bigr)~\in~]w_3,w_4[\,,\qquad \forall~t\in \,]t_0, \Hat \tau_0], $$
where $]w_3,w_4[$ is an arbitrarily small neighborhood of $w^+(u_0)$.

\v
{\bf 4.}  In {\bf CASE 2}, by the previous step,  the profile $u(\Hat\tau_0,\cdot) $
satisfies
\bel{Htu1}  u(\Hat \tau_0, x)~=~\ds
{\alpha\over 2} - {x-X_0\over\beta (\Hat \tau_0-t_0)}~=~ {\alpha\over 2} - \delta(x-X_0)\qquad \hbox{if}\qquad 
x\in \,\bigl]y'(\Hat \tau_0),\,y''(\Hat \tau_0)\bigr[\,,\eeq
\bel{Htu2}
\bigl|u(\Hat \tau_0, x)- u_0\bigr|~<~\delta_1\quad \hbox{if}\quad x\notin \bigl[y'(\Hat \tau_0),\,y''(\Hat \tau_0)\bigr].\eeq
Moreover, for any given $\delta_2>0$, we can choose $\delta_1>0$ small enough so that
\bel{Htu3}\bega{rl} \Big| y'(\Hat\tau_0) -X_0 -  (\Hat\tau_0-t_0) g'\bigl( w^+(u_0)\bigr)\Big|&<~\delta_2\,,\\[2mm]
\Big| y''(\Hat\tau_0) -X_0 - (\Hat\tau_0-t_0) g'\bigl( w^-(u_0)\bigr)\Big|&<~\delta_2\,.\enda\eeq
\bel{Htu4} \Big|u\bigl(t, y'(\Hat\tau_0)+\bigr) - w^+(u_0)\Big|~<~\delta_2\,,\qquad\qquad 
 \Big|u\bigl(t, y''(\Hat\tau_0)-\bigr) - w^-(u_0)\Big|~<~\delta_2\,.\eeq

\begin{figure}[ht]
\centerline{\hbox{\includegraphics[width=14cm]{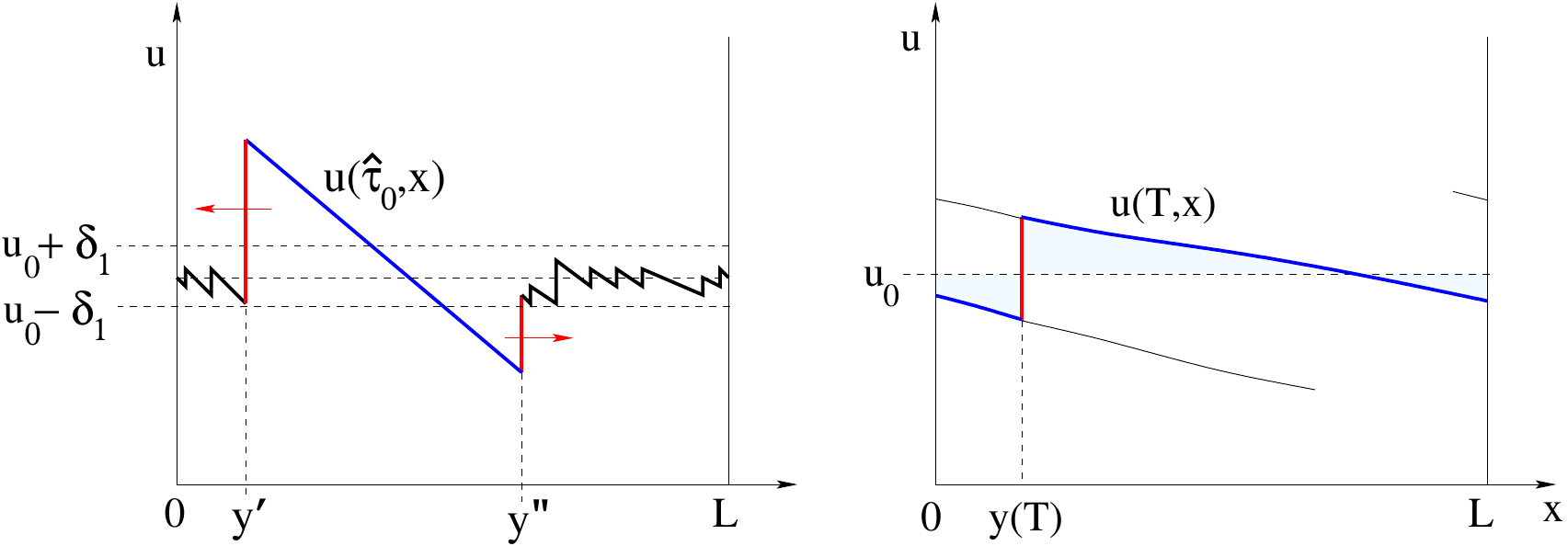}}}
\caption{\small Left: the solution $u((\Hat \tau_0, \cdot)$ at the time 
when the flux switches from $g$ to $f$ 
inside the region enclosed between the two large shocks.  Right: at a later time $T$, the solution contains a single shock, and is $\C^1$ everywhere else. Note that, while the 
implicit equation (\ref{utdef}) (adjusted by periodicity) yields a multivalued function 
on $[0,L]$, the location of the shock is uniquely determined by the identity (\ref{masscons}), accounting for the conservation of the total mass. }
\label{f:df135}
\end{figure}
\v{\bf 5.} Next, consider any solution $u=u(t,x)$ to 
 the scalar conservation law (\ref{fclaw}), with initial data 
 $u(\Hat \tau_0,\cdot)$ assigned on the entire real line, satisfying (\ref{Htu1})--(\ref{Htu4}).
 By the standard theory of scalar conservation laws, always assuming $\delta_1>0$ sufficiently small,
 there exists a uniformly bounded time $T^*>\Hat\tau_0$ at which the two large shock satisfy
$$y''(T^*) - y'(T^*)~=~L.$$
For our original solution $u$, which is $L$-periodic in space,
this means that at time $T^*$ the two shocks merge into a single one (see again Fig.~\ref{f:df136}, left).
\v
{\bf 6.} To prove the last statement of the lemma,
let $U(t,x)$ be the solution of (\ref{fgclaw}) with  a single spike generated at $(t_0, X_0)$, but with 
constant initial data:
\bel{U0}U(0,x)~=~u_0\qquad \qquad\forall x\in [0,L].\eeq
For $t>t_0$, denote by $Y'(t)<Y''(t)$  the two large shocks in $U(t,\cdot)$.

Of course, the two solutions $u,U$ will be different.   However,
in the 
common domain 
 between the two shocks, they coincide. In particular, for $t\in [t_0, \Hat\tau_0]$ we have
 \bel{U=u}   U(t,x)~=~u(t,x)~=~{\alpha\over 2} - {x-X_0\over\beta (t-t_0)}\qquad\qquad\forall x~\in~\bigl]y'(t),\, y''(t)\bigr[\,\cap \,\bigl]Y'(t),\, Y''(t)\bigr[\,.\eeq

Next, recalling that $\Hat\tau_0=t_0+\Hat\tau$  with $\Hat\tau\doteq 1/\beta\delta$,  
consider the solution $u(t,\cdot)$ for $t>\Hat \tau_0$.
The region
enclosed between the two shocks $y'(t)<y''(t)$  is now covered by $f$-characteristics  which at time $ \Hat \tau_0$
start inside the interval $\bigl[ y'( \Hat \tau_0), \, y''( \Hat \tau_0)\bigr]$.
Hence $u(t,\cdot)$ is implicitly determined by
\bel{utdef}
u\bigg(t,~x+ (t-t_0-\Hat\tau) f'\Big(   {\alpha\over 2} -\delta(x-X_0)\Big)\bigg) \,=~ {\alpha\over 2} -\delta(x-X_0),\qquad\quad \hbox{for}\qquad x\in \,\bigl] y'(t),\,y''(t)\bigr[\,.
\eeq
Notice that, for $ x\in \,\bigl] Y'(t),\,Y''(t)\bigr[\,$, the solution $U(t,x)$ is determined exactly by the same implicit formula (\ref{utdef}).

Finally (see Fig.~\ref{f:df136}), after the two shocks merge into a single one (by $L$-periodicity), for $t\geq T^*$ on the whole domain $[0,L]$ both solutions are 
determined by the equation (\ref{utdef}).   At first sight, they may differ for the locations of the 
single shocks $y(t)$ and $Y(t)$.   However, this is not the case, because 
%
%
%
%we see that the function
%$u(t,\cdot)$ at (\ref{utdef}) depends only on two parameters: $t_0 $ and $X_0$.   The initial data $u(\Hat\tau_0,\cdot)$  affect only the location of the two shocks $y'(t), y''(t)$.
%
%
%
%Note that, as the function implicitly defined at (\ref{utdef}) is prolonged by
%$L$-periodicity, it becomes multivalued (see Fig.~\ref{f:df135}, right).  A unique
%single-valued solution is recovered by imposing 
the conservation of total mass:
\bel{masscons} \int_0^L u(T,x)\, dx ~=~\int_0^L U(T,x)\, dx~=~L u_0\,.\eeq
 uniquely determines the position $y(t)=Y(t)$ of the shocks present in $u(t,\cdot)$ and $U(t,\cdot)$, 
 respectively.
 
The above analysis shows that, regardless of the initial data $\bar u\in\S$,
at the  fixed terminal time $T$,  the solution $u(T,\cdot)$ is implicitly defined by 
the identity 
\bel{uTdef}
U\Big(x+ (T-t_0-\Hat\tau) f'\bigl(   {\alpha\over 2} -\delta(x-X_0)\bigr)\Big) ~=~ {\alpha\over 2} -\delta(x-X_0).
\eeq
In view of the identity (\ref{masscons}) that uniquely determines the location of the shock,
these functions depend only on two parameters: the time $t_0$ and the position $X_0$.
\endproof

\begin{remark}\label{r:61}{\rm
For future use, let again $T=T_1+1+T_3$ be as in Lemma~\ref{l:61}.  
On the rectangle $\Gamma= [T_1, \, T_1+1]\times [0,L]$, consider the probability measure $\Tilde \nu$
defined as a constant multiple of  Lebesgue measure. Let $\S$ be the family of all admissible
profiles, introduced in Definition~\ref{d:22}.
Define the map $\Lambda:\Gamma\mapsto \S$  by setting
$$\Lambda(t_0, X_0) ~=~U^{(t_0, X_0)}(T,\cdot).$$
Here $U^{(t_0, X_0)}$ denotes the solution to (\ref{fgclaw}) with constant initial data (\ref{U0})
and containing a single $N$-shaped spike generated at $(t_0, X_0)$.
Denoting by 
\bel{calM}
\M~\doteq~\Big\{ U^{(t_0, X_0)}(T,\cdot)\,;~~ t_0\in [T_1, \, T_1+1],~~X_0 \in  [0,L]\Big\}~\subset~\S
\eeq
the image of this map,
the push-forward of $\Tilde \nu$ by $\Lambda$
yields a probability measure $\nu(\cdot)$ on $\M$. 
}
\end{remark}

\section{Bounds on the number of shocks}
 \label{sec:7}
\setcounter{equation}{0}
In this section we consider a random solution
$t\mapsto u(t,\cdot)$  of (\ref{fgclaw}), (\ref{id}), as constructed in Section~\ref{sec:2}, and call $N(t)$ the corresponding number of shocks (i.e., upward jumps).   Since more and more new spikes are generated, in principle
this number may increase without bound.   However, shocks merge into each other, thus decreasing their total number.   Goal of this section is to derive an upper bound on the expected number of shocks, independent of time.  

\begin{proposition} Let the times $T_1, T_2, T_3$ be as in Lemma~\ref{l:61}, and set
\bel{TEN} T\,\doteq\, T_1+T_2+T_3,\qquad\qquad  \rho\,\doteq\,e^{-\lambda T_1} \cdot \lambda T_2 e^{-\lambda T_2}\cdot e^{-\lambda T_3}\,=\, \lambda e^{-\lambda T}.\eeq
Then, at  any time $t\geq 0$, the expected number of jumps in a
random solution $u(t,\cdot)$ is bounded by
\bel{ENt} E\bigl[N(t)\bigr]~\leq~
\frac{\rho+2\lambda T}{\rho}\,.\eeq
\end{proposition}

{\bf Proof.} {\bf 1.} For any $t\geq 0$ and $n\geq 1$ we have the bound on the conditional expectation
\bel{CEB}
E\Big[ N(t+T)~\Big|~N(t)=n\Big]~\leq~n+ {2\lambda T}\,.\eeq
Indeed, the expected number of random times $t_j\in [t, t+T]$ when a new spike
can be generated
is no larger than $2\lambda T$. Since each spike contains two shocks, this 
yields (\ref{CEB}).  Notice that here we are not taking into account possible cancellations
due to merging of shocks.
\v
{\bf 2.} Denote by $P_{010}$ the event that the number of random times $t_j$ 
contained inside the time intervals
$$\bigl[t, ~t+T_1\bigr], \qquad \bigl[t+T_1,~t+T_1+T_2\bigr], \qquad \bigl[t+T_1+T_2,~t+T_1+T_2+T_3\bigr],$$
 is 0,1,0, respectively.
 Since these three events are mutually independent, the probability that $P_{010}$ occurs is
precisely the constant $\rho$ in (\ref{TEN}).

In this case, by Lemma~\ref{l:61} the function $u(t+T,\cdot)$ 
contains one single upward shock. 
\v
{\bf 3.}
We denote by $P_{010}^c$ the complement,
i.e.~the case where $P_{010}$ fails.  This event occurs with probability $1-\rho$.
We thus have
\bel{prob2}\bega{rl}
E\bigl[N(t+T)\bigr]&=~E\Big[ N(t+T)\,\Big|\, P_{010}\Big]\cdot \prob(P_{010})
+ E\Big[ N(t+T)\,\Big|\, P^c_{010}\Big]\cdot \prob(P^c_{010})\\[3mm]
&\leq~1\cdot \rho +  E\Big[ N(t+T)\,\Big|\, P^c_{010}\Big]\cdot (1-\rho).
\enda\eeq
\v
{\bf 4.} 
Let $K=\#\bigl\{j\geq 1\,;~~t_j\in [t, t+T]\bigr\}$ be the number of random times
occurring within the interval $[t,t+T]$.
Since at each $t_j$ at most two new shocks are generated, we have
\[
N(t+T)\,\leq\, N(t)+2K.
\]
Taking conditional expectations, one obtains
\bel{E3}
E\Big[N(t+T)\,\Big|\,P^c_{010}\Big]
~\leq~
E\Big[N(t)\,\Big|\,P^c_{010}\Big]
+
2E\Big[K\,\Big|\,P^c_{010}\Big].
\eeq
Observing that
$N(t)$ is independent of $P_{010}^c$, we have
\[
E\Big[N(t)\,\Big|\,P^c_{010}\Big]
=
E[N(t)].
\]
Moreover, since $N(t)\geq 0$, we have the bound 
\[
E\bigl[K\mid P^c_{010}\bigr]~
\leq~
\frac{E[K]}{\prob(P^c_{010})}~
=~
\frac{\lambda T}{1-\rho}.
\]
Using the two previous bounds in (\ref{E3}), we now obtain
\[
E\Big[N(t+T)\,\Big|\,P^c_{010}\Big]
~\leq~
E\bigl[N(t)\bigr]
+
\frac{2\lambda T}{1-\rho}\,.
\]
\v
{\bf 5.}
By (\ref{prob2}) it follows
\bel{prob5}
E\bigl[N(t+T)\bigr]
~\leq ~ \rho + \left( E\bigl[ N(t)\bigr] + \frac{2\lambda T}{1-\rho}\right)\cdot (1-\rho).\eeq
The bound (\ref{ENt}) can now be obtained by induction.
For $t\in [0,T]$ we have the easy bound
$$E\bigl[N(t)\bigr]~\leq~2 \,E\Big[ \#\bigl\{ j\geq 1\,;~~t_j\in [0,t]\bigr\}\Big]~\leq~2 \,E\Big[ \#
\bigl\{ j\geq 1\,;~~t_j\in [0,T]\bigr\}\Big]~=~2\lambda T.$$
Introduce the constant
\[
M\,\doteq \,\frac{\rho+2\lambda T}{\rho}\,,
\]
and assume that $E\bigl[N(t)\bigr]~\leq~M$ for $t\in [0, nT]$.
Then, for any $\tau\in \bigl[nT, (n+1)T\bigr]$, by (\ref{prob5}) we have
\[
 E\bigl[N(\tau)\bigr]~\leq~  \rho + \left( E\bigl[ N(\tau-T)\bigr] + \frac{2\lambda T}{1-\rho}\right)\cdot (1-\rho)~\leq~ \rho + \left( M + \frac{2\lambda T}{1-\rho}\right)\cdot (1-\rho)~=~M.
\]
By induction on $n$, this achieves the proof.
\endproof

\section{Ergodicity}
\label{sec:8}
\setcounter{equation}{0}
In this section we study the family of random solutions to (\ref{fgclaw}) 
from the point of view of general 
Markov processes.

We recall the main result from \cite{HM, H}. 
Fix a measurable space $\bfX$ and a Markov transition kernel $\P(x, \cdot)$ on $\bfX$. 
Assume that $\P$ satisfies the following conditions.
\begi
\item[{\bf (C1)}] {\it There exists a function $V: \bfX\mapsto \R_+$  and constants $K\geq 0$ and  $\gamma\in \,]0,1[\,$ such that
\bel{c1}(\P V)(x)~\leq~\gamma V(x) + K\qquad\qquad\forall x\in \bfX.\eeq
}

\item[{\bf (C2)}] {\it 
There exists a constant $\eta\in \,]0,1[\,$  and a probability measure $\nu$  so that 
\bel{c2} \P(x, \cdot)~\geq~\eta \nu(\cdot)\qquad\forall x\in \bfX
~\hbox{such that} ~V (x) \leq  R,\eeq
for some constant 
$R> 2K/(1-\gamma)$.}\endi

To state Harris’ theorem~\cite{H}, we introduce the following weighted supremum norm:
\bel{snorm} \|\vp\|_V  ~\doteq~\sup_x~{\bigl|\vp(x)\bigr|\over 1+V(x)}\,. \eeq

\begin{theorem}\label{t:31}
If the conditions {\bf (C1)-(C2)} hold, then $\P$ admits a unique invariant measure
$\mu_*$. Moreover,  there exist constants $C > 0$ and $\gamma\in \,]0,1[\,$ such that the bounds
\bel{Pnb}
\bigl\| \P^n\vp - \mu_*(\vp)\bigr\|~\leq~C \gamma^n \bigl\| \vp-\mu_*(\vp)\bigr\|\,,\qquad\qquad \forall
n\geq 1,\eeq
hold for every measurable function $\vp : \bfX\mapsto\R$ such that $\|\vp\|<\infty$. 
\end{theorem}

%{\color{red} Following \cite{HM}, write out a proof of this theorem, in the special case we need.}  

We shall only need the above result
for the particular Markov process associated with our stochastic traffic
model. Accordingly, after introducing the state space below, we shall give a
direct proof of the conclusion of Theorem~\ref{t:31} in this special case,
using the stronger estimates available in our setting.

For a fixed $T>0$, we shall first apply the above theorem to the Markov process  
with discrete times: $\tau_n = n\, T$, where the transition kernel $\P$ corresponds to 
the random map
$$u(t,\cdot)~~\mapsto~~u(t+T,\cdot).$$

In view of Theorem~\ref{t:21}, we define $\bfX \doteq \S$ to be the set of admissible profiles 
considered in Definition~\ref{d:22}.
%
%all piecewise  Lipschitz
%continuous functions
%$u:[0,L]\mapsto [0,\alpha]$, with upward jumps at points
%$0\leq y_1<y_2<\cdots< y_n<L$, such that
%\begi
%\item[(I)]  On each open interval $I_k(t)\doteq\,\bigl]y_{k-1}, y_k\bigr[\,,$  either 
%$u(t,x)= u_0$, or else $u_x(t,x)<0$.
%\item[(II)] If $u_x(t,x)\leq -\delta$ at some point $x\in I_k$, then $u_x(t,\cdot)$ is constant
%on $I_k(t)$.
%\endi
For  $u\in \bfX$, we then define
\bel{Vu}
V(u)~\doteq~\TV\bigl\{u\,;~[0,L]\bigr\}.\eeq

\begin{theorem}\label{t:82}
For every 
$T>0$ sufficiently large, the transition kernel $\P$ on $\bfX$ satisfies the two 
conditions {\bf (C1)-(C2)}.   Hence the stochastic process in discrete time:  $\tau_n = n\, T$, 
admits a unique stationary probability distribution $\mu^*$.   Moreover,  the bounds
(\ref{Pnb})  hold.
\end{theorem}

%
%In essence, this follows from the fact that, for any initial data $\bar u\in \bfX$, with positive probability, at some fixed time $T>0$ the random solution takes values within 
%a compact set $\M$ diffeomorphic  to $[0,1]\times [0,L]$ (with endpoints 0 and $L$ identified).
%

{\bf Proof.}  {\bf 1.} Condition {\bf (C1)} is straightforward.   In view of Proposition~\ref{prop:41},
choosing $T$ large enough, for any $u(t,\cdot)\in \bfX$ one has
\bel{TVub}
E\Big[ \tv\bigl\{ u(t+T,\cdot)\bigr\}~\Big|~u(t,\cdot)\Big]~\leq~K~\doteq~
2\alpha \cdot \left( 1+ e^{L/\alpha\beta^*}\right).
\eeq
Hence (\ref{c1}) trivially holds for any $\gamma\geq 0$.
\v
{\bf 2.}
To establish Condition {\bf (C2)}, let $\nu(\cdot)$ be the probability measure on the 2-dimensional 
set $\M$
constructed in Remark~\ref{r:61}.   We claim that, for some $\eta>0$ small enough, there holds
\bel{c3} \P(u, \cdot)~\geq~\eta\, \nu(\cdot)\qquad\forall u\in \bfX.\eeq

Indeed, let $T=T_1+1+T_3$ be as in Lemma`\ref{l:61}.
Consider any initial data $\bar u\in \bfX=\S$.
Then, with positive probability, the random solution 
$u(T,\cdot)$
lies in the 2-dimensional 
manifold $\M$.   
Moreover, as in the proof of Lemma~\ref{l:61}, in this case the profile $u(T,\cdot)$ depends only on the
time $t_0\in [T_1, T_1+1]$ and the location $X_0\in [0,L]$ where the new spike is generated.
 This implies (\ref{c3}), for some $\eta>0$ small enough, independent of 
$\bar u\in \bfX$.
\v
{\bf 3.}
Thus we have established \eqref{TVub}-\eqref{c3}, which are stronger than the assumptions
{\bf (C1)}--{\bf (C2)}. 
For completeness we now give a direct proof that these stronger estimates imply the conclusion of
Theorem~\ref{t:31}.

By \eqref{c3}, for every \(u\in\bfX\) we can write
\[
\P(u,\cdot)=\eta\nu(\cdot)+(1-\eta)Q(u,\cdot),
\]
where \(Q\) is another Markov transition kernel.

Let \(\lambda\) be a signed measure on \(\bfX\) with total mass zero. Since
the common part \(\eta\nu\) cancels, we have
\[
\P\lambda\,=\,(1-\eta)Q\lambda .
\]
Fix $\beta>0$ sufficiently small and define the weighted total variation norm
\[
\|\lambda\|_\beta
\,\doteq\,
\int_{\bfX}(1+\beta V(u))\,|\lambda|(du).
\]
By \eqref{TVub} it follows
\[
(1-\eta)(QV)(u)
~\le~
(\P V)(u)
~\le~ K .
\]
Therefore
\[
\bega{rl}
\|\P\lambda\|_\beta
&=~(1-\eta)\|Q\lambda\|_\beta        \\[3mm]
&\le~\ds
(1-\eta)\int_{\bfX}\bigl(1+\beta(QV)(u)\bigr)\,|\lambda|(du)    \\[4mm]
&\le~\ds
\int_{\bfX}\bigl((1-\eta)+\beta K\bigr)\,|\lambda|(du).
\enda
\]
Choosing \(\beta>0\) small enough we achieve
\bel{Pcontr}
\|\P\lambda\|_\beta\,\le\, \gamma\|\lambda\|_\beta ,
\eeq
for some $\gamma<1$.
Thus \(\P\) is a strict contraction on signed measures of total mass zero.
\v
{\bf 4.}
Fix \(\bar u\in\bfX\) and denote by $\delta_{\bar u}$ the probability measure
concentrating a unit mass at $\bar u$.  Setting
\[
\mu_n\,\doteq\, \P^n\delta_{\bar u}\,,
\]
by induction we obtain
\[
\|\mu_{n+1}-\mu_n\|_\beta
~=~
\|\P^n(\P\delta_{\bar u}-\delta_{\bar u})\|_\beta
~\le~
\gamma^n\|\P\delta_{\bar u}-\delta_{\bar u}\|_\beta\, .
\]
Hence $(\mu_n)_{n\geq 1}$ is a Cauchy sequence in the weighted total variation norm.
Let \(\mu^*\) be its limit. Letting $n\to\infty$ in the identity
\[
\mu_{n+1}=\P\mu_n
\]
one obtains
\[
\P\mu^*=\mu^*.
\]
showing that \(\mu^*\) is invariant.

If \(\mu\) is another invariant probability measure with finite \(V\)-moment, then
\[
\|\mu-\mu^*\|_\beta
~=~
\|\P^n\mu-\P^n\mu^*\|_\beta
~\le~
\gamma^n\|\mu-\mu^*\|_\beta .
\]
Letting \(n\to\infty\), we obtain \(\mu=\mu^*\). Hence the invariant measure
is unique.
\v
{\bf 5.}
Finally, for any measurable \(\psi\) with \(\|\psi\|_V<\infty\),
\[
\left|\P^n\psi(u)-\mu^*(\psi)\right|
~\le~
\|\psi\|_V\,\|\P^n\delta_u-\mu^*\|_\beta
~\le~
C\gamma^n(1+V(u))\|\psi\|_V .
\]
Dividing by \(1+V(u)\) and taking the supremum over \(u\in\bfX\), we obtain
\[
\|\P^n\psi-\mu^*(\psi)\|_V
~\le~
C\gamma^n\|\psi\|_V .
\]
Applying this estimate to
\[
\psi~=~\varphi-\mu^*(\varphi)
\]
one completes the proof of Theorem~\ref{t:82}.
\endproof
\v

 It remains to pass from the discrete-time Markov chain to the continuous time process.
Let \(\{ S_t\,;~t\geq 0\}\) denote the Markov semigroup in  continuous time.
The preceding argument shows that, for some \(T>0\), the transition kernel
\[
 \P=S_T
\]
has a unique invariant probability measure \(\mu^*\).  Moreover, for every 
probability measure $\mu$ on $\bfX$ one has
\bel{limn} \lim_{n\to +\infty} \bigl\|S_{ nT} \,\mu - \mu_*\bigr\| ~=~0.\eeq
We now show that the limit (\ref{limn}) holds more generally as time ranges over all real values.

\begin{theorem} In the above setting, for every probability measure $\mu$ on $\bfX$ one has
\bel{Slim}  \lim_{t\to +\infty} \bigl\|S_t \mu - \mu_*\bigr\| ~=~0.\eeq
\end{theorem}

{\bf Proof.} {\bf 1.} We observe that the same argument used in the proof of (\ref{Pcontr}) actually 
shows that there exists $\gamma\in \,]0,1[\,$ such that 
\bel{Scon}\| S_\tau \lambda\|_\beta~\leq~\gamma \|\lambda\|_\beta\eeq
for every signed measure $\lambda$ of zero total mass and every $\tau\in [T, 2T]$.
\v
{\bf 2.}  Next, we claim that $\mu_*$ is invariant under the entire semigroup $(S_t)_{t\geq
0}$. Indeed, for every $s\geq0$, using the semigroup property and the invariance
of $\mu_*$ under $S_T$, we obtain
\[
S_T(S_s\mu_*)
\,=\,
S_s(S_T\mu_*)
\,=\,
S_s\mu_*\,.
\]
Therefore $S_s\mu_*$ is an invariant probability measure for the transition kernel
$S_T$. By the uniqueness of the $S_T$-invariant probability measure, it follows
that
\bel{mustarall}
S_s\mu_*\, =\, \mu_*
\qquad\hbox{for every }s\geq0.
\eeq
\v
{\bf 3.} To prove (\ref{Slim}) it suffices to show that, for every sequence of times $\tau_j\to +\infty$,
one can extract a subsequence $(\tau_{j(m)})_{m\geq 1}$ such that
\bel{Slim2}  \lim_{m\to\infty} \bigl\|S_{\tau_{j(m)}}  \mu - \mu_*\bigr\| ~=~0.\eeq

 Given a sequence $(\tau_j)_{j\geq1}$ with $\tau_j\to+\infty$, we first extract
a subsequence $(\tau_{j(m)})_{m\geq 1}$ such that
\[
\tau_{j(m+1)}-\tau_{j(m)}\,\geq\, T
\qquad\hbox{for every }m\geq1.
\]
For every interval $[\tau_{j(m)},\tau_{j(m+1)}]$, choose an integer $N_m\geq1$ such that
\[
T~\leq~\frac{\tau_{j(m+1)}-\tau_{j(m)}}{N_m}~\leq~2T.
\]
Dividing each interval
$[\tau_{j(m)},\tau_{j(m+1)}]$ into $N_m$ equal parts and arranging all the resulting
points in increasing order, we obtain a sequence $(\tau'_k)_{k\geq1}$ such that
\begi
\item[(i)] $\tau'_{k+1}-\tau'_k\in [T, 2T]$ for every $k\geq 1$.
\item[(ii)] For every $m\geq 1$ one has  $\tau_{j(m)}\in \{\tau'_k\,;~k \geq 1\}$.
\endi

By (i), each transition kernel $S_{\tau'_{k+1}-\tau'_k}$ is a strict contraction on the
family of signed measures with zero total mass. Since $S_{\tau'_k}\mu-\mu_*$ has zero total mass, (\ref{Scon}) yields
\[
\bigl\|S_{\tau'_{k+1}}\mu-\mu_*\bigr\|_\beta ~=~ \bigl\|S_{\tau'_{k+1}-\tau'_k}\bigl(S_{\tau'_k}\mu-\mu_*\bigr)\bigr\|_\beta
~\leq~
\gamma
\bigl\|S_{\tau'_k}\mu-\mu_*\bigr\|_\beta\,.
\]
Iterating this estimate, we find
\bel{betadecay}
\bigl\|S_{\tau'_k}\mu-\mu_*\bigr\|_\beta
~\leq~
\gamma^{k-1}
\bigl\|S_{\tau'_1}\mu-\mu_*\bigr\|_\beta.
\eeq
We
choose $\tau'_1\geq T$, so that the quantity on the right-hand side of
(\ref{betadecay}) is finite. Since $0<\gamma<1$, and the weighted norm dominates the total variation norm, we conclude
\bel{lim3}
\lim_{k\to\infty}
\bigl\|S_{\tau'_k}\mu-\mu_*\bigr\|=0.
\eeq

By (ii),  for every $m\geq 1,$ there exists $k_m \in \mathbb{N}$ such that $\tau_{j(m)}=\tau_{k_m}'$. Hence (\ref{lim3}) implies (\ref{Slim2}).
\endproof

\section{Numerical simulations}
\label{sec:9}
\setcounter{equation}{0}
To simplify our computations, we first  take 
\bel{fgdef} g(u) \,=\, u-u^2,\qquad\qquad f(u)\,=\, u-u^2 +\ve_0,\eeq 
for some $\ve_0>0$, and then we slightly modify the values of $f$ in a neighborhood of 0 and 1,
so that all the assumptions in {\bf (A1)} hold.   In particular $f(0)=g(0) =0$ and
$f(1)=g(1)=0$.   For $\ve_0>0$ small, we expect that with high probability the
random solution $u=u(t,x)$
will take values  inside the region where both $f$ and $g$ are  quadratic polynomials.

At every time $t>0$, the solution $u(t,\cdot)$ will have a polygonal graph. It will have upward jumps
at finitely many shocks, say   located at 
$$x_1(t) ~<~ x_2(t)~<~\cdots~<~x_N(t),\qquad\qquad x_i(t)\in [0, L[\,,$$
and will be affine on each of the open intervals $]x_{i-1}, x_i[$.
The positions of the points $x_i(t)$ and the values $u_i^\pm(t)$ to the left and right of these points satisfy a system of ODEs.   These ODEs must be be restarted when one of these four
cases happens:

\begi
\item[(i)] By periodicity, when one of the points $x_i(t) $ crosses the end of the interval $[0,L]$,
i.e. when $x(t)=L$,  it is replaced by a shock located at $x(t)=0$ (see Fig.~\ref{f:df107}).

\item[(ii)] When two shocks interact, the solution is restarted by joining them into a single shock.

\item[(iii)] When the gradient $u_x(t,\cdot)<0$, which is constant on each interval 
$\bigl] x_{i-1}(t), x_i(t)\bigr[\,$,
becomes larger than $-\delta$, the flux on that interval is changed from $g$ to $f$.
Notice that, according to (\ref{fgdef}),
\bel{uxt} {d\over dt} u_x(t,\cdot)~=~-g''(u) u_x^2~=~2 u_x^2\,.\eeq

\item[(iv)] At a random time and point $(t_j, X_j)$, if the flux is $f$ a new spike is formed.
The size of this spike depends on the value $u=u(t_j, X_j)$.   It has to be computed 
by finding the points $Q_1$, $Q_2$ where the secants through $P= \bigl(u, f(u)\bigr)$
are tangent to the graph of $g$, see Fig.~\ref{f:df131}.  Notice that a new spike will be produced as soon as $u_x(t_j, X_j)>-\delta$.
\endi

\begin{figure}[ht]
\centerline{\hbox{\includegraphics[width=11cm]{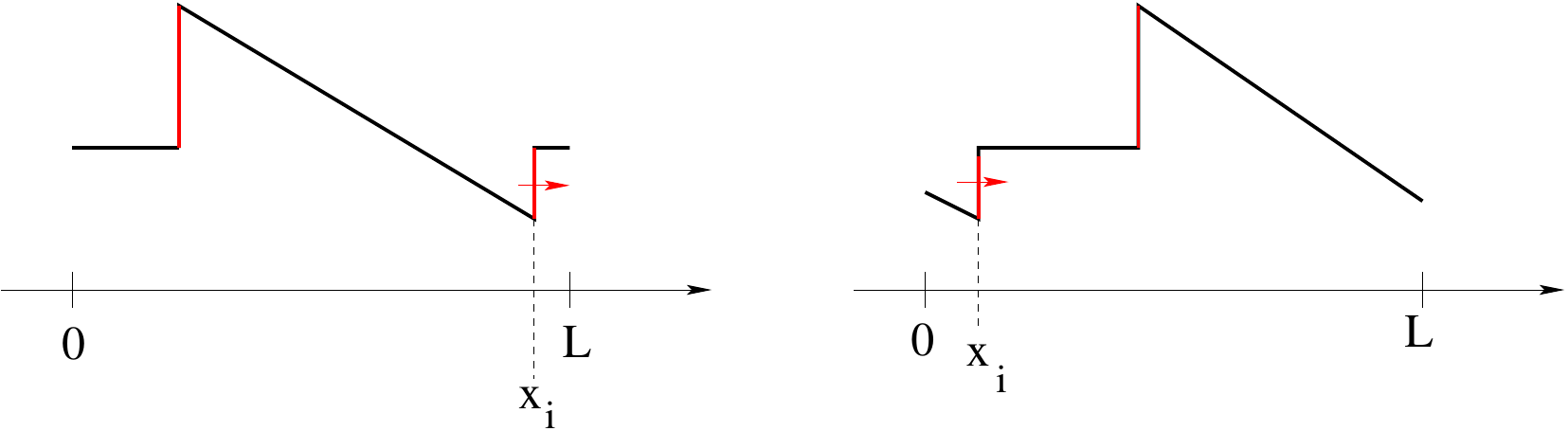}}}
\caption{\small By periodicity, when a shock at $x_i(t)$ crosses the endpoint $x=L$, it is restarted at $x=0$. }
\label{f:df107}
\end{figure}

{\bf Evolution of nodes across a shock.}    Let $(x_k, u_k)$ and $(x_{k+1}, u_{k+1})$ 
be the two nodes along the polygonal graph of $u(t,\cdot)$ which correspond to a shock, with $x_k=x_{k+1}$, $u_k<u_{k+1}$.

The shock speed is determined by the Rankine-Hugoniot equations. 
\bel{RHC}\hbox{shock speed}~=~{\hbox{[flux to the right] - [flux to the left]}\over
\hbox{[state to the right] - [state to the left]}}\eeq
This leads to 4 cases, depending on whether the fluxes to the right and to the left
are given by $f$ or $g$.
$$\dot x_k ~=~\dot x_{k+1}~=~{f(u_{k+1}) - f(u_k)\over u_{k+1}- u_k},\eqno(f,f)$$
$$\dot x_k ~=~\dot x_{k+1}~=~{g(u_{k+1}) - g(u_k)\over u_{k+1}- u_k},\eqno(g,g)$$
$$\dot x_k ~=~\dot x_{k+1}~=~{f(u_{k+1}) - g(u_k)\over u_{k+1}- u_k},\eqno(f,g)$$
$$\dot x_k ~=~\dot x_{k+1}~=~{g(u_{k+1}) - f(u_k)\over u_{k+1}- u_k}.\eqno(g,f)$$
\v
In addition, the states $u_k$, $u_{k+1}$ now vary in time, because characteristics impinge on the shock
from both sides.

In the four above cases we respectively have
$$\dot u_k~=~\bigl[ f'(u_k) - \dot x_k\bigr] \cdot {u_k-u_{k-1}\over x_k-x_{k-1}}\,,\qquad
\dot u_{k+1}~=~\bigl[ \dot x_{k+1} - f'(u_{k+1}) \bigr] \cdot {u_{k+2}-u_{k+1}\over x_{k+2}-x_{k+1}}\,,\qquad\eqno(f,f)$$
$$\dot u_k~=~\bigl[ g'(u_k) - \dot x_k\bigr] \cdot {u_k-u_{k-1}\over x_k-x_{k-1}}\,,\qquad
\dot u_{k+1}~=~\bigl[ \dot x_{k+1} - g'(u_{k+1}) \bigr] \cdot {u_{k+2}-u_{k+1}\over x_{k+2}-x_{k+1}}\,,\qquad\eqno(g,g)$$
$$\dot u_k~=~\bigl[ f'(u_k) - \dot x_k\bigr] \cdot {u_k-u_{k-1}\over x_k-x_{k-1}}\,,\qquad
\dot u_{k+1}~=~\bigl[ \dot x_{k+1} - g'(u_{k+1}) \bigr] \cdot {u_{k+2}-u_{k+1}\over x_{k+2}-x_{k+1}}\,,\qquad\eqno(f,g)$$
$$\dot u_k~=~\bigl[ g'(u_k) - \dot x_k\bigr] \cdot {u_k-u_{k-1}\over x_k-x_{k-1}}\,,\qquad
\dot u_{k+1}~=~\bigl[ \dot x_{k+1} - f'(u_{k+1}) \bigr] \cdot {u_{k+2}-u_{k+1}\over x_{k+2}-x_{k+1}}\,.\qquad\eqno(g,f)$$

Here $\dot x_k = \dot x_{k+1}$ is the speed of the shock, given by the previous formulas.
Note that here  the first factor is the difference in speed, i.e.~the rate at which
characteristics flow into the shock.   The second factor is the slope of the solution to the left 
or to the right of the shock.

By (\ref{uxt})  the slopes of the segments where the solution is decreasing
can be explicitly computed.
Calling
\bel{zdef}z_{k-1}(t)~\doteq~ {u_k-u_{k-1}\over x_k-x_{k-1}}\,,\qquad\qquad 
z_{k+1}(t)~\doteq~ {u_{k+2}-u_{k+1}\over x_{k+2}-x_{k+1}}\,,\eeq
we have
$${d\over dt} z_{k-1}(t)~=~ 2 z^2_{k-1}(t),\qquad\qquad {d\over dt} z_{k+1}(t)~=~ 2 z^2_{k+1}(t)
.$$
If $t_j$ is one of the restarting times, then for $t\in [t_j, t_{j+1}]$ we have
\bel{zf} z_{k-1}(t)~=~{1\over {1\over z_{k-1}(t_j)} - 2(t-t_j)} \,,
\qquad\qquad  z_{k+1}(t)~=~{1\over {1\over z_{k+1}(t_j)} - 2(t-t_j)} \,.\eeq
Since $z_{k-1}(t_j)<0$,~$ z_{k+1}(t_j)<0$, this means that as $t\to +\infty$ these slopes increase to zero.

\begin{figure}[ht]
\centerline{\hbox{\includegraphics[width=14cm]{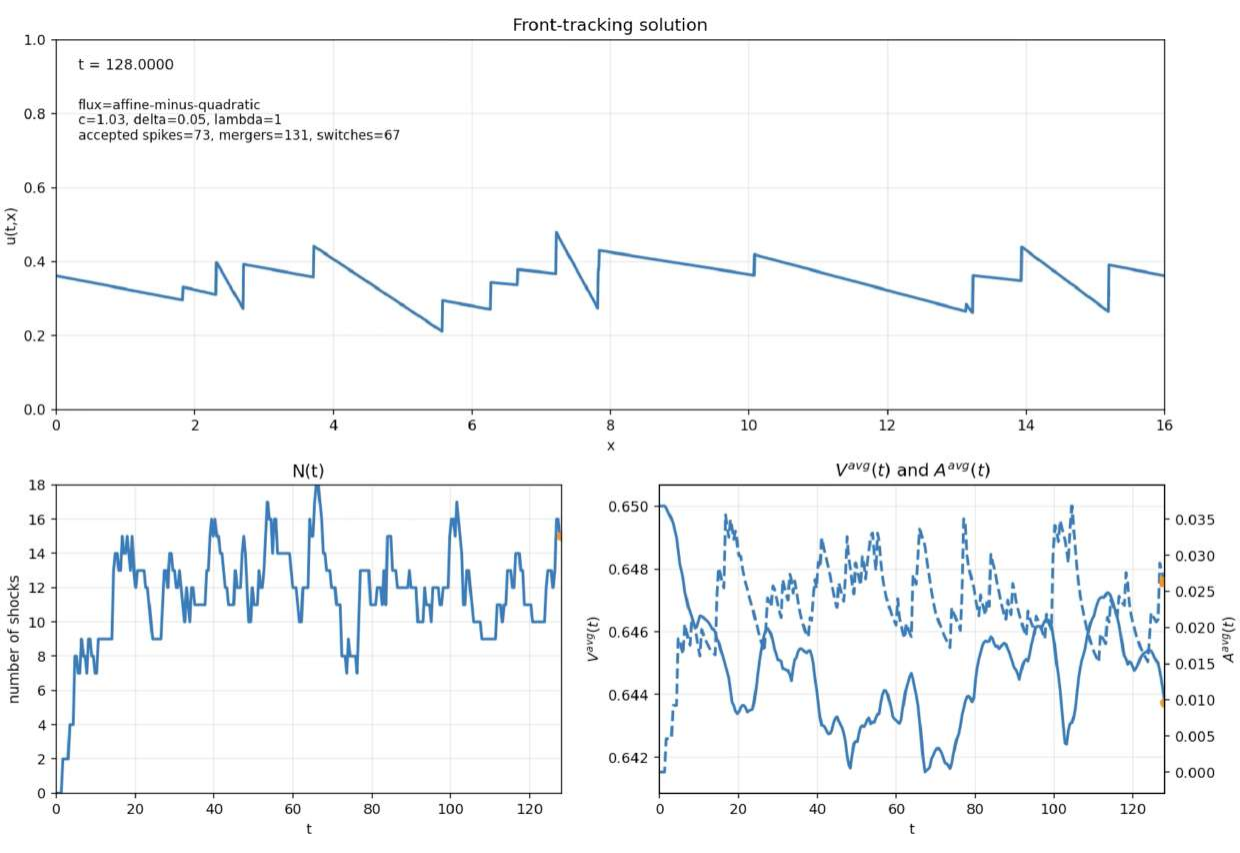}}}
\caption{\small Top: the profile of a random solution $u(\tau,\cdot)$ at a time $\tau>\!>1$.
Bottom left: the number of shocks $N(t)$, as a function of time.    Bottom right:
The average velocity of cars $V^{avg}(t)$ (solid line) and the average acceleration of cars
$A^{avg}(t)$ (dotted line).  Note that, when new spikes are formed, so that $N(t)$ increases, the average velocity decreases and the average acceleration increases. }
\label{f:tsnap}
\end{figure}

Note that, if  a new $N$-shaped spike is generated at $(t_i, X_i)$, then for $t>t_i$
the graph of $u(t,\cdot)$ will contain a segment with slope
\bel{dslope}u_x~=~{-1\over 2(t-t_i)}\,.\eeq
Indeed, this is the interval occupied by the centered rarefaction wave originating at $(t_i, X_i)$.
The formula (\ref{dslope}) can be used to obtain an explicit solution to (\ref{zf}).

The left and right states $u_k(t)$ and $u_{k+1}(t)$ are now computed by solving the appropriate system of ODEs corresponding to one of the four cases above, with coefficients depending on $t.$
In turn, 
recalling \eqref{RHC},  the corresponding shock position
$x_k(t)=x_{k+1}(t)$ is then obtained by solving the ODE
\bel{doxk}
\dot x_k(t)~=~
{\hbox{[flux to the right] - [flux to the left]}\over
u_{k+1}(t)-u_k(t)}\,,
\eeq
Here the flux on the right can be $f\bigl(u_{k+1}(t)\bigr)$ or $g(u_{k+1}(t)\bigr)$, and 
the flux on the left can be $f\bigl(u_{k}(t)\bigr)$ or $g(u_{k}(t)\bigr)$, in different cases. 
\v
Finally, the algorithm must be restarted at times when two shocks interact, and at random times where a new $N$-shaped spike is generated.
Each restarting requires:
\begi
\item Deleting  two nodes, at a time when two shocks merge into a single one.
\item Inserting four new nodes, at random times when a new spike arises.\endi
In both cases, all the nodes need to be relabeled.

\section{Concluding remarks}
\label{sec:10}
\setcounter{equation}{0}
In this paper we introduced a simple model, based on two flux functions, which generates a rich
stochastic dynamics.   By suitably adjusting the parameters,  some
features of actual traffic flow can be recovered.  

More general models could be considered. In particular:
\begi 
\item At (\ref{P}) the rate $\lambda$  at which new spikes are generated, instead of being a constant,  may depend on the density $u$.
\item The graphs of two functions $f$ and $g$ may intersect at some intermediate point $u^*\in \,]0,\alpha[\,$, 
say,
$$\left\{\bega{rl} f(u)<g(u)\quad &\hbox{if}~~0<u<u^*,\\[2mm]
g(u)<f(u)\quad &\hbox{if}~~u^*<u<\alpha.\enda\right.$$
\endi
A difficult open problem is to understand what happens if the flux $g$ is not a quadratic polynomial.
In this direction we expect that,
if the assumption (\ref{gprop}) is dropped, 
our previous analysis should remain valid provided that the equation (\ref{fgclaw}) is replaced
by
\bel{fgclaw2}
u_t + \Big[\theta(u_x) f(u) + \bigl(1-\theta(u_x) \bigr)g(u)\Big]_x=~0,\qquad \quad
\theta(s)~=~\left\{\bega{rl} 1\quad &\hbox{if}\quad s>-\delta(u),\cr
0\quad &\hbox{if}\quad s<-\delta(u),\enda\right.\eeq
\bel{du}\hbox{where}\quad \delta(u)=-{c\over  g''(u)}\quad \hbox{ for some constant}~ c>0.\eeq
Notice that this choice of $\delta(u)$  guarantees that the flux function will switch from $g$ to $f$ simultaneously at all points in the interior of an $N$-shaped spike.
Apart from the very special case (\ref{du}), it is not clear how to construct solutions to (\ref{fgclaw2}).

\v
{\bf Acknowledgments.} This research was partially supported by NSF with
grant  DMS-2306926, ``Regularity and approximation of solutions to conservation laws".
The authors wish to thank Yuri Suhov  for valuable discussions and suggestions, which were instrumental in establishing the ergodicity of the stochastic process.
They also wish to thank Rinaldo Colombo for his help in writing the Python code used in the numerical simulations, with the aid of chatGPT.

\v


\begin{thebibliography}{99}
\bibitem{ABS1}  D.~Amadori, A.~Bressan and W.~Shen,
Conservation laws with discontinuous gradient-dependent flux: the stable case. 
 {\it Math. Models Methods Appl. Sci.} {\bf 35} (2025), 1421--1469.


\bibitem{ABS2}  D.\,Amadori, A.\,Bressan and W.\,Shen,
Conservation laws with discontinuous gradient-dependent flux: the unstable case. 
{\it Comm. Math. Sci.} {\bf 24} (2026), 1971--1997.


\bibitem{APT}  A.\,Hayat, B.\,Piccoli and S.\,Truong, 
Dissipation of traffic jams using a single autonomous vehicle on a ring road.
{\it SIAM J. Appl. Math.} {\bf 83} (2023), 909--937.


\bibitem{Bbook} A.\,Bressan,
{\it Hyperbolic systems of conservation laws. The one dimensional Cauchy problem}. 
Oxford University Press, 2000.


\bibitem{BMPR} M.\,Briani, R.\,Manzo, B.\,Piccoli and L.\,Rarit\`a,
Estimation of  $NO_x$  and  $O_3$  reduction by dissipating traffic waves
{\it Netw. Heterog. Media} {\bf 19} (2024), 822--841.




\bibitem{CB} V.\,Capasso and D.\,Bakstein,
{\it An Introduction to Continuous-Time Stochastic Processes: Theory, Models, and Applications to Finance, Biology, and Medicine}. Modeling and Simulation in Science, Engineering and Technology,
4-th edition, 2021.


\bibitem{CPT} F.\,A.\,Chiarello, B.\,Piccoli and A.\,Tosin,
Multiscale control of generic second order traffic models by driver-assist vehicles.
{\it Multiscale Model. Simul. } {\bf 19} (2021), 589--611.




\bibitem{CDGP} M.\,T.\,Chiri, C.\,Denaro, X.\,Gong and B.\,Piccoli, 
Hybrid multi-population traffic flow model: optimal control for a mean-field limit.
{\it J. Hyperbolic Differ. Equat.} {\bf 23} (2026), 127--150.




\bibitem{CF19} A.\,Corli and H.\,Fan,
Hysteresis and stop-and-go waves in traffic flows.
{\it Math.\,Models Methods Appl. Sci.}, {\bf  29} (2019), 2637--2678.


\bibitem{CF23}
A.\,Corli and H.\,Fan, String stability in traffic flows. {\it Appl. Math. Comput.} {\bf 443} (2023), Paper No.\,127775, 24.


\bibitem{CF24}
A.\,Corli and H.\,Fan, The hysteretic Aw-Rascle-Zhang model. {\it Stud. Appl. Math.}, {\bf 153}, (2024), Paper No.\,e12769, 38.


\bibitem{CF25}
A.\,Corli and H.\,Fan, String stability for hysteretic flows. Preprint, 2026.


\bibitem{F24}
H.\,Fan, Conservation laws with hysteretic fluxes. {\it J.\,Differential Equations} 
{\bf 394} (2024), 1--30.


\bibitem{FS} H.\,Fan and W.\,Shen, Microscopic hysteretic traffic model and stop-and-go 
waves. {\it Discr. Contin. Dyn. Syst.} {\bf 45} (2025), 4490--4509.


\bibitem{GGLP} M.\,Garavello, P.\,Goatin, T.\,Liard and B.\,Piccoli, 
A multiscale model for traffic regulation via autonomous vehicles.
{\it J. Differential Equations} {\bf 269} (2020), 6088--6124.


%
%\bibitem{G} P.\,Goatin,
%The Aw–Rascle vehicular traffic flow model with phase transitions, 
%{\it Math. Computer Modeling}  {\bf 44} (2006), 287--303.


\bibitem{Seib}
M.\,R.\,Flynn, A.\,R.\,Kasimov, J.-C.\,Nave, R.\,R.\,Rosales and B.\,Seibold, 
Self-sustained nonlinear waves in traffic flow.
{\it Phys. Rev. E} (3) {\bf 79} (2009), 056113.


%\bibitem{GR92} T.\,Gimse and  N.\,H.\,Risebro.
%Solution of the Cauchy problem for a conservation law with discontinuous flux function.
%{\it SIAM J.\,Math.\,Anal.}, \textbf{23} (1992), 635--648.


\bibitem{HM} M.\,Hairer and J.\,C.\,Mattingly, Yet another look at Harris’ ergodic theorem for Markov chains. In:
{\it Seminar on Stochastic Analysis, Random Fields and Applications VI}.
R.\,C.\,Dalang, M.\,Dozzi,
and F.\.Russo, Editors.
Springer, Basel, 2011, pp.\,109--117. 
 
\bibitem{H} T.\,E.\,Harris, The existence of stationary measures for certain Markov processes. In {\it Proceedings of the Third Berkeley Symposium on Mathematical Statistics and Probability, 1954--1955, vol.\,II}, 113--124. University of California Press,  1956.
 
 

\bibitem{HR}
H.\,Holden and N.\,Risebro, 
{\it Front Tracking for Hyperbolic Conservation Laws.} Second Edition.  Springer-Verlag, Berlin, 2015.

 
\bibitem{Li} T.\,Li,
Nonlinear dynamics of traffic jams.
{\it Phys. D} {\bf  207} (2005), 41--51.
    
\bibitem{LW}
M.\,Lighthill and G.\,Whitham,  On kinematic waves. II. A theory of traffic flow on long crowded roads. 
{\it Proceedings of the Royal Society of London: Series A,}
{\bf 229} (1955), 317--345.




\bibitem{RRS} R.\,Ramadan, R.\,R.\,Rosales and B.\,Seibold,
         Structural properties of the stability of jamitons.  In:
{\it Mathematical descriptions of traffic flow: micro, macro and
              kinetic models}.
ICIAM 2019 SEMA SIMAI Springer Ser., {\bf 12},
 pp.\,35--62.




\bibitem{R}
P.\,I.\,Richards, Shock waves on the highway, 
{\it Oper. Res.} {\bf  4} 
(1956),  42-51.


\bibitem{XXX} R.\,E.\,Stern, S.\,Cui, M.\,L.\,Delle Monache, R.\,Bhadani, M.\,Bunting, 
M.\,Churchill, N.\,Hamilton, H.\,Pohlmann, F.\,Wu, B.\,Piccoli, B.\,Seibold, J.\,Sprinkle, 
and D.\,B.\,Work, Dissipation of stop-and-go waves via control of autonomous vehicles: Field experiments, {\it Trans. Res. Part C Emerg. Technol.}, {\bf 89}  (2018), 205--221. 




\end{thebibliography}
\end{document}